\documentclass[12pt]{amsart}
\usepackage{amssymb}

\usepackage{amsmath,amsfonts,euscript,amscd,amsthm,amssymb,upref,graphics}

\theoremstyle{plain}
\newtheorem{theorem}{Theorem}
\swapnumbers

\newtheorem{proposition}[subsection]{Proposition}
\newtheorem{lemma}[subsection]{Lemma}
\newtheorem{corollary}[subsection]{Corollary}

\newtheorem{notation}[subsection]{Notation}

\theoremstyle{definition}
\newtheorem{definition}[subsection]{Definition}
\newtheorem{example}[subsection]{Example}
\newtheorem{remark}[subsection]{Remark}

\newtheorem{nothing*}[subsection]{}
\newcommand{\rien}[1]{}

\newcommand{\AVF}{ \operatorname{{\rm AVF}}}
\newcommand{\VFH}{ \operatorname{{\rm VF}_{hol}}}

\newcommand{\LieH}{ \operatorname{{\rm Lie}_{hol}}}
\newcommand{\VFA}{ \operatorname{{\rm VF}_{alg}}}
\newcommand{\LieA}{ \operatorname{{\rm Lie}_{alg}}}

\newcommand{\Aut}{ \operatorname{{\rm Aut}}}

\newcommand{\C}{\ensuremath{\mathbb{C}}}
\newcommand{\T}{\ensuremath{\mathbb{T}}}
\newcommand{\R}{\ensuremath{\mathbb{R}}}

\newcommand{\Z}{\ensuremath{\mathbb{Z}}}

\newcommand{\hgoth}{{\ensuremath{\mathfrak{h}}}}

\newcommand{\sgoth}{{\ensuremath{\mathfrak{s}}}}
\newcommand{\lgoth}{{\ensuremath{\mathfrak{l}}}}

\newcommand{\rgoth}{{\ensuremath{\mathfrak{r}}}}

\newcommand{\ggoth}{{\ensuremath{\mathfrak{g}}}}

\newcommand{\cR}{{\ensuremath{\mathcal{R}}}}
\newcommand{\Ker}{{\rm Ker} \,}

\renewcommand{\epsilon}{\varepsilon}
\renewcommand{\phi}{\varphi}

\begin{document}
\title[Algebraic density property of homogeneous spaces ]{Algebraic density property of homogeneous spaces}
\author{Fabrizio Donzelli}
\address{Department of Mathematics\\
University of Miami\\
Coral Gables, FL 33124 \ \ USA}
\email{f.donzelli@math.miami.edu}
\author{Alexander Dvorsky}
\address{Department of Mathematics\\
University of Miami\\
Coral Gables, FL 33124 \ \ USA}
\email{a.dvorsky@math.miami.edu}
\author{Shulim Kaliman}
\address{Department of Mathematics\\
University of Miami\\
Coral Gables, FL 33124 \ \ USA}
\email{kaliman@math.miami.edu}

\begin{abstract}
Let $X$ be an affine algebraic variety with a transitive action of the
algebraic automorphism group. Suppose that $X$ is equipped with several
fixed point free non-degenerate $SL_2$-actions satisfying some mild
additional assumption. Then we prove that the Lie algebra generated by
completely integrable algebraic vector fields on $X$ coincides with the set
of all algebraic vector fields. In particular, we show that apart from a few
exceptions this fact is true for any homogeneous space of form $G/R$ where $%
G $ is a linear algebraic group and $R$ is its proper reductive subgroup.
\end{abstract}

\maketitle

{\renewcommand{\thefootnote}{}
}

\vfuzz=2pt

\vfuzz=2pt

\section{Introduction}

In this paper we develop further methods introduced by F. Kutzschebauch and
the third author in \cite{KK2} which they used to obtain new results in the
Anders\' en-Lempert theory (\cite{A}, \cite{AL}). D. Varolin was the first
one to realize the importance of the following notion in that theory (\cite
{V1}).

\begin{definition}
\label{dense}

A complex manifold $X$ has the density property if in the
compact-open topology the Lie algebra $\LieH (X)$ generated by
completely integrable holomorphic vector fields on $X$ is dense in
the Lie algebra $\VFH (X)$ of all holomorphic vector fields on
$X$. An affine algebraic manifold $X$ has the algebraic density
property if the Lie algebra $\LieA (X)$ generated by completely
integrable algebraic vector fields on it coincides with the Lie
algebra $\VFA (X)$ of all algebraic vector fields on it (clearly,
the algebraic density property implies the density property).
\end{definition}

For any complex manifold with the density property the Anders\' en-Lempert
theory is applicable and its effectiveness in complex analysis was
demonstrated in several papers (e.g., see \cite{FR}, \cite{V1}, \cite{V2}).
Furthermore, Koll\'ar and Mangolte found a remarkable application of this
theory to real algebraic geometry \cite{KoMa}. However until recently the
class of manifolds for which this property was established was quite narrow
(mostly Euclidean spaces and semi-simple Lie groups, and homogeneous spaces
of semi-simple groups with trivial centers \cite{TV1}, \cite{TV2}). In \cite
{KK1} this class was enlarged by hypersurfaces of form $uv= p (\bar x)$ and
in \cite{KK2} by connected complex algebraic groups except for $%
\ensuremath{\mathbb{C}}_+$, $\ensuremath{\mathbb{C}}^*$ (for which the
density property is not true) and the higher dimensional tori (for which the
validity of this property is still unknown). Furthermore, it was shown in
\cite{KK1}, \cite{KK2} that these varieties have the algebraic density
property.

In this paper we study a smooth complex affine algebraic variety
$X$ with a transitive action of the algebraic automorphism group
$\Aut X$ (which is natural because complex manifolds with the
algebraic density property have transitive automorphism groups).
Though the facts we prove about such objects are rather
straightforward extension of \cite{KK2}, in combination with Lie
group theory they lead to a much wider class of homogeneous spaces
with the algebraic density property. Our new technique yields, in
particular, to the following.

\textbf{Theorem.} \emph{Let $G$ be a linear algebraic group and $R$ be its
proper reductive subgroup such that the homogeneous space $G/R$ is different
from $\ensuremath{\mathbb{C}}_+$, a torus, or a $\ensuremath{\mathbb{Q}}$%
-homology plane ${\ensuremath{\mathcal{P}}}$ with a fundamental group $%
\ensuremath{\mathbb{Z}}_2$. Then $G/R$ has the algebraic density property.}

Besides the criteria developed in \cite{KK2} the main new ingredient of the
proof is the Luna slice theorem. For convenience of readers we remind it in
Section 2 together with basic facts about algebraic quotients and some
crucial results from \cite{KK2}. In Section 3 we prove our main theorem. As
an application we obtain the Theorem before in section 4 using some
technical fact from the Lie group theory presented in the Appendix.

\emph{Acknowledgments.} We would like to thank Lev Kapitanski,
Frank Kutzschebauch, and William M. McGovern for inspiring
discussions and consultations.

\section{Preliminaries}

Let us fix some notation first. In this paper $X$ will be always a complex
affine algebraic variety and $G$ be an algebraic group acting on $X$, i.e. $%
X $ is a $G$-variety. The ring of regular functions on $X$ will be denoted
by $\ensuremath{\mathbb{C}} [X ]$ and its subring of $G$-invariant functions
by $\ensuremath{\mathbb{C}} [X]^G$.

\subsection{Algebraic (categorical) quotients.}

Recall that the algebraic quotient $X//G$ of $X$ with respect to the $G$%
-action is $Spec(\ensuremath{\mathbb{C}} [X]^G)$. By $\pi : X \to X//G$ we
denote the natural quotient morphism generated by the embedding $%
\ensuremath{\mathbb{C}} [X]^G \hookrightarrow \ensuremath{\mathbb{C}} [X]$.
The main (universal) property of algebraic quotients is that any morphism
from $X$ constant on orbits of $G$ factors through $\pi$. In the case of a
reductive $G$ several important facts (e.g., see \cite{Sch}, \cite{PV}, \cite
{D}, \cite{G}) are collected in the following.

\begin{proposition}
\label{redu} Let $G$ be a reductive group.

\textrm{(1)} The quotient $X//G$ is an affine algebraic variety which is
normal in the case of a normal $X$ and the quotient morphism $\pi
:X\rightarrow X//G$ is surjective.

\textrm{(2)} The closure of every $G$-orbit contains a unique closed orbit
and each fiber $\pi^{-1} (y)$ (where $y \in X//G$) contains also a unique
closed orbit $O$. Furthermore, $\pi^{-1} (y)$ is the union of all those
orbits whose closures contain $O$.

\textrm{(3)} In particular, if every orbit of the $G$-action on $X$ is
closed then $X//G$ is isomorphic to the orbit space $X/G$.

\textrm{(4)} The image of a closed $G$-invariant subset under $\pi$ is
closed.
\end{proposition}

If $X$ is a complex algebraic group, and $G$ is a closed subgroup acting on $X$
by multiplication,
clearly all the orbits are closed. If $G$ is reductive, the previous
proposition implies that the quotient $X/G$ is affine. The next proposition
(Matsushima's criterion) shows that the converse is also true.

\begin{proposition}
\label{matsu} Let $G$ be a complex reductive group, and $H$ be a closed
subgroup of $G$. Then the quotient space $G/H$ is affine if and only if $H$
is reductive.
\end{proposition}

Besides reductive groups actions in this paper, a crucial role
will be played by $\ensuremath{\mathbb{C}}_+$-actions. In general
algebraic quotients in this case are not affine but only
quasi-affine \cite{W}. However, we
shall use later the fact that for the natural action of any $\ensuremath{\mathbb{C}}%
_+$-subgroup of $SL_2$ generated by multiplication one has $SL_2//%
\ensuremath{\mathbb{C}}_+\cong \ensuremath{\mathbb{C}}^2$.

\subsection{Luna's slice theorem (e.g., see (\protect\cite{D},\protect\cite
{PV}).}

%
%
Let us remind some terminology first. Suppose that $f :X\rightarrow Y$ is a $%
G$-equivariant morphism of affine algebraic $G$-varieties $X$ and $Y$. Then
the induced morphism $f_G:X//G\rightarrow Y//G$ is well defined and the
following diagram is commutative.

\begin{equation}
\begin{CD} X @>f >> Y\\ @VVV @VVV\\ X//G @>f_G >> Y//G \end{CD}
\end{equation}

\hspace{0.2cm}

\begin{definition}
A $G$-equivariant morphism $f$ is called strongly \' etale if

(1) The induced morphism $f_G: X//G\rightarrow Y//G $ is \' etale

(2) The quotient morphism $\pi_G :X\rightarrow X//G$ induces a $G$-
isomorphism between $X$ and the fibred product $Y\times_{Y//G}(X//G)$.
\end{definition}

From the properties of \'etale maps (\cite{D}) it follows that $f$ is \'
etale (in particular, quasi-finite).

Let $H$ be an algebraic subgroup of $G$, and $Z$ an affine $H$-variety. We
denote $G\times_H Z$ the quotient of $G \times Z$ by the action of $H$ given
by $h(g,z)=(gh^{-1},hz)$. The left multiplication on $G$ generates a left
action on $G\times_H Z$. The next lemma is an obvious consequence of \ref
{redu}.

\begin{lemma}
\label{twist} Let $X$ be an affine $G$-variety and $G$ be reductive. Then
the $H$-orbits of $G\times X$ are all isomorphic to $H$. Therefore the
fibers of the quotient morphism $G\times X\rightarrow G\times_H X$ coincide
with the $H$-orbits.
\end{lemma}

The isotropy group of a point $x\in X$ will be denoted by $G_x$. Recall also
that an open set $U$ of $X$ is called saturated if $\pi_G^{-1}(\pi_G(U))=U$.
We are ready to state the Luna slice theorem.

\begin{theorem}
\label{Luna} Let $G$ be a reductive group acting on an affine algebraic
variety $X$, and let $x\in X$ be a point in a closed $G$-orbit. Then
there exists a locally closed affine algebraic subvariety $V$ (called a
slice) of $X$ containing $x$ such that

(1) $V$ is $G_x$-invariant;

(2) the image of the $G$-morphism $\phi : G\times_{G_x}V$ induced by the
action is a saturated open set $U$ of $X$;

(3) the restriction $\phi : G\times_{G_x}V\rightarrow U$ is strongly \'
etale. \hspace{0.1cm}

\end{theorem}

Given a saturated open set $U$, we will denote $\pi_G(U)$ by $U//G$. It
follows from \ref{redu} that $U//G$ is open. 
This theorem implies that the following diagram is commutative

\begin{equation}  \label{moon}
\begin{CD} G\times_{G_x}V @>>> U\\ @VVV @VVV\\ V//G_x @>>> U//G \end{CD}
\end{equation}
and $G\times_{G_x}V \simeq U \times_{U//G} V//G_x$.

%
%

\subsection{The compatibility criterion}

This section presents the criteria for the algebraic density property,
introduced in \cite{KK2}, that will be used to prove the main results of
this paper.


%


\begin{definition}
Let $X$ be an affine algebraic manifold. An algebraic vector field $\sigma$
on $X$ is semi-simple if its phase flow is an algebraic $\ensuremath{%
\mathbb{C}}^*$-action on $X$. A vector field $\delta$ is locally nilpotent
if its phase flow is an algebraic $\ensuremath{\mathbb{C}}_+$-action on $X$.
In the last case $\delta$ can be viewed as a locally nilpotent derivation on
$\ensuremath{\mathbb{C}} [X ]$. That is, for every nonzero $f \in %
\ensuremath{\mathbb{C}} [X]$ there is the smallest $n=n(f)$ for which $%
\delta^n (f)=0$. We set $\deg_{\delta} (f)=n-1$. In particular, elements
from the kernel $\mathrm{Ker} \, \delta$ have the zero degree with respect
to $\delta$.
\end{definition}


\begin{definition}
\label{compatible} Let $\delta_1$ and $\delta_2$ be nontrivial algebraic
vector fields on an affine algebraic manifold $X$ such that $\delta_1$ is a
locally nilpotent derivation on $\ensuremath{\mathbb{C}} [X]$, and $\delta_2$
is either also locally nilpotent or semi-simple. That is, $\delta_i$
generates an algebraic action of $H_i$ on $X$ where $H_1 \simeq %
\ensuremath{\mathbb{C}}_+$ and $H_2$ is either $\ensuremath{\mathbb{C}}_+$
or $\ensuremath{\mathbb{C}}^*$. We say that $\delta_1$ and $\delta_2$ are
semi-compatible if the vector space $\mathrm{Span} (\mathrm{Ker} \, \delta_1
\cdot \mathrm{Ker} \, \delta_2)$ generated by elements from $\mathrm{Ker} \,
\delta_1 \cdot \mathrm{Ker} \, \delta_2$ contains a nonzero ideal in $%
\ensuremath{\mathbb{C}} [X]$.

A semi-compatible pair is called compatible if in addition one of the
following condition holds

(1) when $H_2\simeq \ensuremath{\mathbb{C}}^*$ there is an element $a \in
\mathrm{Ker} \, \delta_2$ such that $\deg_{\delta_1} (a)=1$, i.e. $\delta_1
(a) \in \mathrm{Ker} \, \delta_1 \setminus \{ 0 \}$;

(2) when $H_2 \simeq \ensuremath{\mathbb{C}}_+$ (i.e. both $\delta_1$ and $%
\delta_2$ are locally nilpotent) there is an element $a$ such that $%
\deg_{\delta_1} (a)=1$ and $\deg_{\delta_2} (a)\leq 1$.
\end{definition}

\begin{remark}
\label{commutative} If $[\delta_1 , \delta_2]=0$ then condition (1) and
condition (2) with $a \in \mathrm{Ker} \, \delta_2$ hold automatically.
\end{remark}

\begin{example}
\label{example.01} Consider $SL_2$ (or even $PSL_2$) with two natural $%
\ensuremath{\mathbb{C}}_+$-subgroups: namely, the subgroup $H_1$ (resp. $H_2$%
) of the lower (resp. upper) triangular unipotent matrices. Denote by
\begin{equation*}
A = \left(
\begin{array}{cc}
a_1 & a_2   \\
b_1 & b_2
\end{array}
\right)
\end{equation*}
an element of $SL_2$. Then the left multiplication generate actions of $H_1$
and $H_2$ on $SL_2$ with the following associated locally nilpotent
derivations on $\ensuremath{\mathbb{C}} [SL_2]$
\begin{equation*}
\delta_1=a_1\frac{\partial}{\partial b_1}+a_2\frac{\partial}{\partial b_2}
\end{equation*}
\begin{equation*}
\delta_2=b_1\frac{\partial}{\partial a_1}+b_2\frac{\partial}{\partial a_2}
\, .
\end{equation*}

Clearly, $\mathrm{Ker} \, \delta_1$ is generated by $a_1$ and $a_2$ while $%
\mathrm{Ker} \, \delta_2$ is generated by $b_1$ and $b_2$. Hence $\delta_1$
and $\delta_2$ are semi-compatible. Furthermore, taking $a=a_1b_2$ we see
that condition (2) of Definition \ref{compatible} holds, i.e. they are
compatible.

\end{example}

It is worth mentioning the following geometrical reformulation of
semi-compatibility which will be needed further.

\begin{proposition}
\label{reformulation} Suppose that $H_1$ and $H_2$ are as in Definition \ref
{compatible}, $X$ is a normal affine algebraic variety equipped with
nontrivial algebraic $H_i$-actions where $i=1,2$ (in particular, each $H_i$
generates an algebraic vector field $\delta_i$ on $X$). Let $X_i=X//H_i$ and
$\rho_i : X \to X_i$ the quotient morphisms. Set $\rho = (\rho_1, \rho_2) :
X \to Y:= X_1 \times X_2$ and $Z$ equal to the closure of $\rho (X)$ in $Y$.
Then $\delta_1$ and $\delta_2$ are semi-compatible iff $\rho : X \to Z$ is a
finite birational morphism.
\end{proposition}

\begin{definition}
\label{tangential} A finite subset $M$ of the tangent space $T_xX$ at a
point $x$ of a complex algebraic manifold $X$ is called a generating set if
the image of $M$ under the action of the isotropy group (of algebraic
automorphisms) of $x$ generates $T_xX$.
\end{definition}

It was shown in \cite{KK2} that the existence of a pair of compatible
derivations $\delta_1$ and $\delta_2$ from Definition \ref{compatible}
implies that $\LieA (X)$ contains a $\ensuremath{%
\mathbb{C}}[X]$-submodule $I \delta_2$ where $I$ is a nontrivial ideal in $%
\ensuremath{\mathbb{C}}[X]$\footnote{
In the case of condition (2) in Definition \ref{compatible} this fact was
proven in \cite{KK2} only for $\deg_{\delta_2} (a)=0$ but the proof works
for $\deg_{\delta_2} (a)=1$ as well without any change.}.  This yields the central
criterion for algebraic density property  \cite
{KK2}.

\begin{theorem}
\label{density} Let $X$ be a smooth homogeneous (with respect to
$\Aut X$) affine algebraic manifold with finitely many pairs of
compatible vector fields $\{ \delta_1^k , \delta_2^k \}_{k=1}^m$
such that for some point $x_0 \in X$ vectors $\{ \delta_2^k (x_0)
\}_{k=1}^m$ form a generating
set. Then $\LieA (X)$ contains a nontrivial $%
\ensuremath{\mathbb{C}}
[X]$-module and $X$ has the algebraic density property.
\end{theorem}

As an application of this theorem we have the following.

\begin{proposition}
\label{prod.11} Let $X_1$ and $X_2$ be smooth homogeneous (with respect to
algebraic automorphism groups) affine algebraic varieties such that each $%
X_i $ admits a finite number of integrable algebraic vector fields $\{
\delta_i^k \}_{k=1}^{m_i}$ whose values at some point $x_i\in X_i$ form a
generating set and, furthermore, in the case of $X_1$ these vector fields
are locally nilpotent. Then $X_1 \times X_2$ has the algebraic density
property.
\end{proposition}

We shall need also two technical results (Lemmas 3.6 and 3.7 in \cite{KK2})
that describe conditions under which quasi-finite morphisms preserve
semi-compatibility.

\begin{lemma}
\label{finereformulation} Let $G=SL_2$ and $X,X^{\prime}$ be normal affine
algebraic varieties equipped with non-degenerate $G$-actions. Suppose that
subgroups $H_1$ and $H_2$ of $G$ are as in Example \ref{example.01}, i.e.
they act naturally on $X$ and $X^{\prime}$. Let $\rho_i : X\to X_i:=X//H_i$
and $\rho_i^{\prime}: X^{\prime}\to X_i^{\prime}:=X^{\prime}//H_i$ be the
quotient morphisms and let $p : X \to X^{\prime}$ be a finite $G$%
-equivariant morphism, i.e. we have commutative diagrams
\begin{equation*}
\begin{array}{ccc}
X & \overset{\rho_i}{\rightarrow} & X_i \\
\, \, \, \downarrow p &  & \, \, \, \, \downarrow q_i \\
X^{\prime} & \overset{\rho_i^{\prime}}{\rightarrow} & X_i^{\prime}
\end{array}
\end{equation*}
\noindent for $i=1,2$. Treat $\ensuremath{\mathbb{C}} [X_i]$ (resp. $%
\ensuremath{\mathbb{C}} [X_i']$) as a subalgebra of $%
\ensuremath{\mathbb{C}}
[X]$ (resp. $\ensuremath{\mathbb{C}} [X']$). Let $\mathrm{Span} (%
\ensuremath{\mathbb{C}}
[X_1] \cdot \ensuremath{\mathbb{C}} [X_2])$ contain a nonzero ideal of $%
\ensuremath{\mathbb{C}}[X]$. Then $\mathrm{Span} (%
\ensuremath{\mathbb{C}}
[X_1'] \cdot \ensuremath{\mathbb{C}} [X_2'])$ contains a nonzero ideal of $%
\ensuremath{\mathbb{C}} [X']$.
\end{lemma}

The second result is presented here in a slightly different form but with a
much simpler proof.

\begin{lemma}
\label{quasi} Let the assumption of Lemma \ref{finereformulation} hold with
two exceptions: we do not assume that $G$-actions are non-degenerate and
instead of the finiteness of $p$ we suppose that there are a surjective
\'etale morphism $r : M \to M^{\prime}$ of normal affine algebraic
varieties equipped with trivial $G$-actions and a surjective $G$-equivariant
morphism $\tau^{\prime}: X^{\prime}\to M^{\prime}$ such that $X$ is
isomorphic to fibred product $X^{\prime}\times_{M^{\prime}}M$ with $p : X\to
X^{\prime}$ being the natural projection (i.e. $p$ is surjective
\'etale). Then the conclusion of Lemma \ref{finereformulation} remains
valid.
\end{lemma}

\begin{proof}

By construction, $X_i =X_i' \times_{M'} M$. Thus we have the
following commutative diagram
\[ \begin{array}{ccccccc}
X & \stackrel{{ \rho}}{\rightarrow} & ({X}_1'\times
{X}_2')\times_{(M' \times M')} (M \times M) & \stackrel{{ 
}}{\rightarrow}
 & M \times M\\
\, \, \, \, \downarrow {p} &&
\, \,  \, \, \downarrow {q} &&  \, \,  \, \, \,  \, \, \, \, \, \downarrow {(r, r)}\\
X' & \stackrel{\rho'}{\rightarrow} &
X_1'\times X_2'  \,  \, & \stackrel{{(\tau' , \tau' )}}{\rightarrow} & \, \, M'\times M'.\\
\end{array} \]
Set $Z$ (resp. $Z'$) equal to the closure of $\rho (X)$ in
$X_1\times X_2$ (resp.  $\rho' (X')$ in $X_1'\times X_2'$) and $D
\simeq M $ (resp. $D' \simeq M'$) be the diagonal subset in
$M\times M$ (resp. $M' \times M')$. 
Since $X= X' \times_{M'} M$ we see that $Z=Z' \times_{D'} D$. For
any affine algebraic variety $Y$ denote by $Y_{\rm norm}$ its
normalization, i.e. $Z_{\rm norm} =Z_{\rm norm}' \times_{D'} D$.
By Lemma \ref{reformulation} $\rho : X \to Z_{\rm norm}$ is an
isomorphism. Since $r$ is surjective it can happen only when
$\rho' : X' \to Z'$ is an isomorphism. Hence the desired
conclusion follows from Lemma \ref{reformulation}.

\end{proof}

The last result from \cite{KK2} that we need allows us to switch from local
to global compatibility.

\begin{proposition}
\label{glue} Let $X$ be an $SL_2$-variety with associated locally nilpotent
derivations $\delta_1$ and $\delta_2$, $Y$ be a normal affine algebraic
variety equipped with a trivial $SL_2$-action, and $r : X \to Y$ be a
surjective $SL_2$-equivariant morphism. Suppose that for any $y \in Y$ there
exists an \'etale neighborhood $g: W \to Y$ such that the vector fields
induced by $\delta_1$ and $\delta_2$ on the fibred product $X\times_Y W$ are
semi-compatible. Then $\delta_1$ and $\delta_2$ are semi-compatible.
\end{proposition}

\section{Algebraic density property and $SL_2$-actions}

\begin{notation}
\label{nota.1} {\rm We suppose that $H_1,H_2, \delta_1$ and
$\delta_2$ are as in Example \ref{example.01}. Note that if $SL_2$
acts algebraically
on an affine algebraic variety $X$ then we have automatically the $%
\ensuremath{\mathbb{C}}_+$-actions of $H_1$ and $H_2$ on $X$ that generate
locally nilpotent vector fields on $X$, which by abuse of notation will be
denoted by the same symbols $\delta_1$ and $\delta_2$. If $X$ admits several
(say, $N$) $SL_2$-actions, we denote by $\{ \delta_1^k,\delta^k_2 \}_{k=1}^N$
the corresponding collection of pairs of locally nilpotent derivations on $%
\ensuremath{\mathbb{C}} [X]$.}
\end{notation}

Here is the first main result of this paper.

\begin{theorem}
\label{MAIN} Let $X$ be a smooth complex affine algebraic variety whose
group of algebraic automorphisms is transitive. Suppose that $X$ is equipped
with $N$ fixed point free non-degenerate actions of $SL_2$-groups $\Gamma_1,
\ldots , \Gamma_N$. Let $\{ \delta_1^k,\delta^k_2 \}_{k=1}^N$ be the
corresponding pairs of locally nilpotent vector fields. If $\{\delta^k_2
(x_0) \}_{k=1}^N \subset T_{x_0}X$ is a generating set at some point $x_0\in
X$ then $X$ has the algebraic density property.
\end{theorem}

\begin{remark}
\label{rem.01} Note that we can choose any nilpotent element of the Lie
algebra of $SL_2$ as $\delta_2$. Since the space of nilpotent elements
generate the whole Lie algebra we can reformulate Theorem \ref{MAIN} as
follows: a smooth complex affine algebraic variety $X$ with a transitive
group of algebraic automorphisms has the algebraic density property provided
it admits ``sufficiently many" fixed point free non-degenerate $SL_2$%
-actions, where ``sufficiently many" means that at some point $x_0 \in X$ the
tangent spaces of the corresponding $SL_2$-orbits through $x_0$ generate the
whole space $T_{x_0}X$.
\end{remark}

By virtue of Theorem \ref{density} the main result will be a consequence of
the following.

\begin{theorem}
\label{goal} Let $X$ be a smooth complex affine algebraic variety equipped
with a fixed point free non-degenerate $SL_2$-action that induces a pair of
locally nilpotent vector fields $\{ \delta_1,\delta_2 \}$. Then these vector
fields are compatible.
\end{theorem}

The proof of the last fact requires some preparations and until we finish
this proof completely the assumption is that all $SL_2$-actions we consider
are \textbf{non-degenerate}.

\begin{lemma}
\label{isosl2} Let the assumption of Theorem \ref{goal} hold and $x \in X$
be a point contained in a closed $SL_2$-orbit. Then the isotropy group of $x$
is either finite, or isomorphic to the diagonal $\ensuremath{\mathbb{C}}^*$%
-subgroup of $SL_2$, or to the normalizer of this $\ensuremath{\mathbb{C}}^*$%
-subgroup (which is the extension of $\ensuremath{\mathbb{C}}^*$ by $\mathbb{%
Z}_2$).

\begin{proof} By Matsushima's criterion (Proposition \ref{matsu})
the isotropy group must be reductive and it cannot be $SL_2$
itself since the action has no fixed points. The only
two-dimensional reductive group is $\C^*\times\C^*$ (\cite{FuHa})
which is not contained in $SL_2$. Thus besides finite subgroups we
are left to consider the one-dimensional reductive subgroups that
include $\C^*$ (which can be considered to be the diagonal
subgroup since all tori are conjugated) and its finite extensions.
The normalizer of $\C^*$ which is its extension by $\mathbb{Z}_2$
generated by
$$A=
  \left(\begin{array}{ccc}
0 & -1\\
1 & 0 \end{array}\right) $$ is reductive.  If we try to find an
extension of $\C^*$ by another finite subgroup that contains an
element $B$ not from the normalizer then $\C^*$ and $B\C^*B^{-1}$
meet at the identical matrix. In particular, the reductive
subgroup must be at least two-dimensional, and we have to
disregard this case.
\end{proof}
\end{lemma}

\begin{proposition}
\label{a} Let $X, \delta_1, \delta_2$ be as in Theorem \ref{goal}. Then
there exists a regular function $g \in \ensuremath{\mathbb{C}} [X]$ such
that $\deg_{\delta_1} (g) =\deg_{\delta_2} (g)=1$.

\begin{proof}

Let $x\in X$ be a point of a closed $SL_2$-orbit. Luna's slice
Theorem yields diagram (\ref{moon}) with $G=SL_2$ and $G_x$ being
one of the subgroups described in Lemma \ref{isosl2}. That is, we
have the natural morphism $\phi: SL_2\times V\rightarrow U$ that
factors through the \'etale morphism $SL_2\times_{G_x}
V\rightarrow U$ where $V$ is the slice at $x$. First, consider the
case when $G_x$ is finite. Then $\varphi$ itself is \'etale.
Furthermore, replacing $V$ by its Zariski open subset and $U$ by
the corresponding Zariski open $SL_2$-invariant subset one can
suppose that $\varphi$ is also finite. Set $f=a_1b_2$ where
$a_i,b_i$ are as in Example \ref{example.01}. Note that each
$\delta_i$ generates a natural locally nilpotent vector field
$\tilde \delta_i$ on $SL_2 \times V$ such that $\C [V] \subset
\Ker \tilde \delta_i$ and $\varphi_* (\tilde \delta_i)$ coincides
with the vector field induced by $\delta_i$ on $X$. Treating $f$
as an element of $\C [SL_2 \times V]$ we have $\deg_{\tilde
\delta_i} (f) = 1, \, i=1,2$. For every $h \in \C [SL_2 \times V
]$ we define a function $\hat h \in \C [U]$ by $\hat h (u)=
\sum_{y \in \varphi^{-1} (u)} h (y)$. One can check that if $h \in
\Ker \tilde \delta_i$ then $\delta_i (\hat h)=0$. Hence
$\delta_i^2 (\hat f) =0$ but we also need $\delta_i (\hat f)\ne 0$
which is not necessarily true. Thus multiply $f$ by $\beta \in \C
[V]$. Since $\beta \in \Ker \tilde \delta_i$ we have $\delta_i
(\widehat {\beta f})(u)=\sum_{y \in \varphi^{-1} (u)} \beta
(\pi_V(y)) \tilde \delta_i (f) (y)$. Note that $\tilde \delta_i
(f) (y_0)$ is not zero at a general $y_0 \in SL_2 \times V$ since
$\tilde \delta_i (f) \ne 0$. By a standard application of the
Nullstellensatz we can choose $\beta$ with prescribed values at
the finite set $\varphi^{-1} (u_0)$ where $u_0=\varphi (y_0)$.
Hence we can assure that $\delta_i ({\widehat {\beta f}}) (u_0)\ne
0$, i.e. $\deg_{ \delta_i} ({\widehat {\beta f}})=1$. There is
still one problem: ${\widehat {\beta f}}$ is regular on $U$ but
necessarily not on $X$. In order to fix it we set $g= \alpha
{\widehat {\beta f}}$ where $\alpha$ is a lift of a nonzero
function on $X//G$ that vanishes with high multiplicity on $(X//G)
\setminus (U//G)$. Since $\alpha \in \Ker \delta_i$ we still have
$\deg_{\delta_i} (g)=1$ which concludes the proof in the case of a
finite isotropy group.

For a one-dimensional isotropy group note that $f$ is
$\C^*$-invariant with respect to the action of the diagonal
subgroup of $SL_2$. That is, $f$ can be viewed as a function on
$SL_2 \times_{\C^*} V$. Then we can replace morphism $\varphi$
with morphism $\psi: SL_2\times_{\C^*} V\rightarrow U$ that
factors through the \'etale morphism $SL_2\times_{G_x}
V\rightarrow U$. Now $\psi$ is also \'etale and the rest of the
argument remains the same.

\end{proof}
\end{proposition}

In order to finish the proof of Theorem \ref{goal} we need to show
semi-compatibility of vector fields $\delta_1$ and $\delta_2$ on $X$. Let $U$
be a saturated set as in diagram (\ref{moon}) with $G=SL_2$. Since $U$ is $%
SL_2$-invariant it is $H_i$-invariant (where $H_i$ is from Notation \ref
{nota.1}) and the restriction of $\delta_i$ to $U$ is a locally nilpotent
vector field which we denote again by the same letter. Furthermore, the
closure of any $SL_2$-orbit $O$ contains a closed orbit, i.e. $O$ is
contained in an open set like $U$ and, therefore, $X$ can be covered by a
finite collections of such open sets. Thus Proposition \ref{glue} implies
the following.

\begin{lemma}
\label{local} If for every $U$ as before the locally nilpotent vector fields
$\delta_1$ and $\delta_2$ are semi-compatible on $U$ then they are
semi-compatible on $X$.
\end{lemma}

\begin{notation}
\label{nota.2} {\rm Suppose further that $H_1$ and $H_2$ act on $%
SL_2\times V$ by left multiplication on first factor. The locally nilpotent
vector fields associated with these actions of $H_1$ and $H_2$ are,
obviously, semi-compatible since they are compatible on $SL_2$ (see Example
\ref{example.01}). Consider the $SL_2$-equivariant morphism $G\times
V\rightarrow G\times_{G_x}V$ where $V$, $G=SL_2$, and $G_x$ are as in
diagram (\ref{moon}). By definition $G\times_{G_x}V$ is the quotient of $%
G\times V$ with respect to the $G_x$-action whose restriction to the first
factor is the multiplication from the right. Hence $H_i$-action commutes
with $G_x$-action and, therefore, one has the induced $H_i$-action on $%
G\times_{G_{x}}V$. Following the patten of Notation \ref{nota.1} we denote
the associated locally nilpotent derivations on $G\times_{G_{x}}V$ again by $%
\delta_1$ and $\delta_2$. That is, the $SL_2$-equivariant \'etale morphism $%
\phi : G\times_{G_{x}}V\rightarrow U$ transforms vector field $\delta_i$ on $%
G\times_{G_{x}}V$ into vector field $\delta_i$ on $U$.}
\end{notation}

From Lemma \ref{quasi} and Luna's slice theorem we have immediately the
following.

\begin{lemma}
\label{product} \textrm{(1)} If the locally nilpotent vector fields $%
\delta_1 $ and $\delta_2$ are semi-compatible on $G\times_{G_{x}}V$ then
they are semi-compatible on $U$.

\textrm{(2)} Furthermore, if the isotropy group $G_x$ is finite $\delta_1$
and $\delta_2$ are, indeed, semi-compatible on $G\times_{G_{x}}V$.
\end{lemma}

Now we have to tackle semi-compatibility in the case of one-dimensional
isotropy subgroup $G_x$ using Proposition \ref{reformulation} as a main
tool. We start with the case of $G_x = \ensuremath{\mathbb{C}}^*$.

\begin{notation}
\label{nota.3} {\rm Consider the diagonal $\ensuremath{\mathbb{C}}^*$%
-subgroup of $SL_2$, i.e. elements of form
\begin{equation*}
s_{\lambda}= \left(
\begin{array}{ccc}
\lambda^{-1} & 0 &  \\
0 & \lambda &
\end{array}
\right) \, .
\end{equation*}
The action of $s_{\lambda}$ on $v\in V$ will be denoted by $\lambda .v$.
When we speak later about the $\ensuremath{\mathbb{C}}^*$-action on $V$ we
mean exactly this action. Set $Y=SL_2\times V$, $Y^{\prime}=SL_2\times_{%
\ensuremath{\mathbb{C}}^*} V$, $Y_i=Y//H_i$, $Y^{\prime}_i=Y^{\prime}//H_i$.
Denote by $\rho_i : Y\rightarrow Y_i$ the quotient morphism of the $H_i$%
-action and use the similar notation for $Y^{\prime}$, $Y_i^{\prime}$. Set $%
\rho =(\rho_1 , \rho_2) : Y \to Y_1 \times Y_2$ and $\rho^{\prime}=(\rho_1^{%
\prime}, \rho_2^{\prime}) : Y^{\prime}\to Y_1^{\prime}\times Y_2^{\prime}$. }
\end{notation}

Note that $Y_i \simeq \ensuremath{\mathbb{C}}^2 \times V$ since $SL_2//%
\ensuremath{\mathbb{C}}_+\simeq \ensuremath{\mathbb{C}}^2$. Furthermore,
looking at the kernels of $\delta_1$ and $\delta_2$ from Example \ref
{example.01} we see for
\begin{equation*}
A= \left(
\begin{array}{ccc}
a_1 & a_2 &  \\
b_1 & b_2 &
\end{array}
\right) \in SL_2
\end{equation*}
the quotient maps $SL_2\rightarrow SL_2//H_1 \simeq \ensuremath{\mathbb{C}}%
^2 $ and $SL_2\rightarrow SL_2//H_2 \simeq \ensuremath{\mathbb{C}}^2$ are
given by $A \mapsto (a_1,a_2)$ and $A \mapsto (b_1,b_2)$
respectively. Hence morphism $\rho : SL_2 \times V = Y \to Y_1\times Y_2
\simeq \ensuremath{\mathbb{C}}^4 \times V \times V$ is given by
\begin{equation}  \label{rho}
\rho (a_1,a_2,b_1,b_2,v)= (a_1,a_2,b_1,b_2,v,v)\, .
\end{equation}
As we mentioned before, to define $Y^{\prime}=SL_2\times_{%
\ensuremath{\mathbb{C}}^*}V$ we let $\ensuremath{\mathbb{C}}^*$ act on $SL_2$
via right multiplication. Since $H_1$ and $H_2$ act on $SL_2$ from the left,
there are well-defined $\ensuremath{\mathbb{C}}^*$-actions on $Y_1$ and on $%
Y_2$ and a torus $\ensuremath{\mathbb{T}}$-action on $Y_1\times Y_2$, where $%
\ensuremath{\mathbb{T}} =\ensuremath{\mathbb{C}}^*\times \ensuremath{%
\mathbb{C}}^*$. Namely,
\begin{equation}  \label{torus}
(\lambda , \mu ).(a_1,a_2,b_1,b_2,v,w)=(\lambda a_1,\lambda^{-1} a_2,\mu
b_1,\mu^{-1} b_2,\lambda .v,\mu .w)
\end{equation}
for $(a_1,a_2,b_1,b_2,v,w)\in Y_1\times Y_2$ and $(\lambda , \mu )\in %
\ensuremath{\mathbb{T}}$.

Since the $\ensuremath{\mathbb{C}}^*$-action on $Y$ and the action of $H_i,
\, i=1,2$ are commutative, the following diagram is also commutative.

\begin{equation}  \label{diag}
\begin{CD} Y @> \rho >> Y_1 \times Y_2\\ @VVpV @VV q V\\ Y' @> \rho ' >>
Y'_1\times Y'_2, \end{CD}
\end{equation}
where $q$ (resp. $p$) is the quotient map with respect to the $%
\ensuremath{\mathbb{T}}$-action (resp. $\ensuremath{\mathbb{C}}^*$-action).
It is also worth mentioning that the $\ensuremath{\mathbb{C}}^*$-action on $%
Y $ induces the action of the diagonal of $\ensuremath{\mathbb{T}}$ on $\rho
(Y)$, i.e. for every $y \in Y$ we have $\rho (\lambda .y) = (\lambda ,
\lambda ). \rho (y)$.

\begin{lemma}
\label{easy} Let $Z=\rho (Y)$ in diagram (\ref{diag}) and $Z^{\prime}$ be
the closure of ${\rho ^{\prime}(Y^{\prime})}$.

\textrm{(i)} The map $\rho:Y\rightarrow Z$ is an isomorphism and $Z$ is the
closed subvariety of $Y_1\times Y_2= \ensuremath{\mathbb{C}}^4 \times V
\times V$ that consists of points $(a_1,a_2,b_1,b_2,v,w)\in Y_1\times Y_2$
satisfying the equations $a_1b_2-a_2b_1=1$ and $v=w$.

\textrm{(ii)} Let $T$ be the $\ensuremath{\mathbb{T}}$-orbit of $Z$ in $%
Y_1\times Y_2$ and $\bar T$ be its closure. Then $T$ coincides with the $(%
\ensuremath{\mathbb{C}}^*\times 1)$-orbit (resp. $(1 \times \C^*)$-orbit) of $Z$.
Furthermore, for each $(a_1,a_2,b_1,b_2,v,w) \in \bar T$ one has $\pi
(v)=\pi (w)$ where $\pi : V \to V//\ensuremath{\mathbb{C}}^*$ is the
quotient morphism.

\textrm{(iii)} The restriction of diagram (\ref{diag}) yields the following
\begin{equation}
\begin{CD}\label{YZ} Y @> \rho >> Z \, \, \, \subset \, \, \, \, \, \bar T\\
@VV p V \, \, \, \, \, \, \, \, \, \, \, \, \, \, \, \, @ VV q V\\ Y' @>
\rho' >> q(Z) \subset Z' \end{CD}\
\end{equation}
where $Y^{\prime}=Y//\ensuremath{\mathbb{C}}^*=Y/\ensuremath{\mathbb{C}}^*$,
$q$ is the quotient morphism of the $\ensuremath{\mathbb{T}}$-action (i.e. $%
Z^{\prime}= \bar T//\ensuremath{\mathbb{T}}$), and $q(Z)=\rho^{\prime}(Y^{%
\prime})$.
\end{lemma}

\begin{proof} The first statement is an immediate consequence
of formula (\ref{rho}). The beginning of the second statement
follows from the fact that the action of the diagonal
$\C^*$-subgroup of $\T$ preserves $Z$. This implies that for every
$t=(a_1,a_2,b_1,b_2,v,w) \in T$ points $v,w \in V$ belong to the
same $\C^*$-orbit and, in particular, $\pi (v) = \pi (w)$. This
equality holds for each point in $\bar T$ by continuity.

In diagram (\ref{diag}) $Y'=Y//\C^*=Y/\C^*$ because of Proposition
\ref{redu} (3) and Lemma \ref{twist}, and the equality $
q(Z)=\rho' (Y')$ is the consequence of the commutativity of that
diagram. Note that $\bar T$ is $\T$-invariant. Hence $q(\bar T)$
coincides with $Z'$ by Proposition \ref{redu} (4). Being the
restriction of the quotient morphism, $q|_{\bar T} : \bar T \to
Z'$ is a quotient morphism itself (e.g., see \cite{D}) which
concludes the proof.

\end{proof}

\begin{lemma}
\label{function} There is a rational $\ensuremath{\mathbb{T}}$%
-quasi-invariant function $f$ on $\bar T$ such that for $%
t=(a_1,a_2,b_1,b_2,w,v) \in T$ one has

\textrm{(1)} ${\frac {1} { f( t)}} a_1b_2-f(t)a_2b_1=1$ and $w=f( t).v$;

\textrm{(2)}
the set $\bar T \setminus T$ is contained in $(f)_0\cup (f)_{\infty}$;

\textrm{(3)} $f$ generates a regular function on a normalization $T_N$ of $T$
\end{lemma}

\begin{proof}
By Lemma \ref{easy} (ii) any point $t= (a_1,a_2,b_1,b_2,w,v) \in
T$ is of form $t= (\lambda , 1) .z_0$ where $z_0 \in Z$ and
$\lambda \in \C^*$. Hence formula (\ref{torus}) implies that
$w=\lambda.v$ and $\lambda^{-1}a_1b_2-\lambda a_2b_1=1$. The last
equality yields two possible values (one of which can be $\infty$
or $0$ if any of numbers $a_1,a_2,b_1$, or $b_2$ vanish)
$$\lambda_\pm = {\frac {-1\pm \sqrt{ 1 +4a_1a_2b_1b_2}}
{2a_2b_1}}$$ and we assume that
$$\lambda =\lambda_- ={\frac {-1- \sqrt{ 1 +4a_1a_2b_1b_2}}
{2a_2b_1}},$$ i.e. $w =\lambda_- .v$. Note that $\lambda_+ .v =w$
as well only when $$\tau ={\frac{\lambda_+}{\lambda_-}}= {\frac {-1+ \sqrt{ 1
+4a_1a_2b_1b_2}}{-1- \sqrt{ 1 +4a_1a_2b_1b_2}}}$$ is in the
isotropy group of $v$.

Consider the set of points $t\in T$ such that $v$ is not a fixed point of the $\C^*$-action on
$V$ and $\tau .
v=v$.  Denote its closure by $S$. Since $S$ is a proper subvariety of $T$, one has a well-defined branch $\lambda_-$ of the two-valued function
$\lambda_\pm$ on the complement to $S$. Its extension
to $\bar T$, which is denoted by $f$, satisfies (1).

Let $t_n\in T$ and $t_n \to t \in \bar T$ as $n \to \infty$. By
Lemma \ref{easy} (ii) $t_n$ is of form $t_n=(f(t_n)a_1^n,{\frac
{1} { f(t_n)}}a_2^n,b_1^n,b_2^n, f(t_n).v_n,v_n)$ where $$
\left(\begin{array}{ccc}
a_1^n & a_2^n\\
b_1^n & b_2^n \end{array}\right) \in SL_2  \, \, \, {\rm and } \,
\,  v=\lim_{n \to \infty} v_n \, .
$$ If sequences $\{f (t_n) \}$ and $\{ 1/f(t_n) \}$ are bounded then
switching to a subsequence one can suppose that  $f(t_n) \to f(t)
\in \C^*$, $w=f(t)v$, and $t=(f(t)a_1',{\frac {1} {
f(t)}}a_2',b_1',b_2', f(t).v,v)$ where
$$\left(\begin{array}{ccc}
a_1' & a_2'\\
b_1' & b_2' \end{array}\right) \in SL_2 \, , $$ i.e. $t \in T$.
Hence $\bar T \setminus T$ is contained in  $ ((f)_0 \cup (f)_{\infty})$ which is (2).

Function $f$ is regular  on $T\setminus S$ by construction. Consider $t\in S$ with
$w$ and $v$ in the same non-constant $\C^*$-orbit, i.e. $w=\lambda .
v$ for some $\lambda \in \C^*$. Then $w=\lambda' .v$ if and only if only $\lambda'$ belongs to the coset $\Gamma$ of the isotropy subgroup  of $v$ in $\C^*$. For any sequence of points $t_n$ convergent to $t$
one can check that  $f(t_n) \to \lambda
\in \Gamma$ by continuity, i.e. $f$ is bounded in a neighborhood of $t$. Let $\nu : T_N \to T$ be a normalization morphism. Then
function $f \circ \nu$ extends regularly to $\nu^{-1} (t)$ by the Riemann extension
theorem. The set of point of $S$ for which $v$ is a fixed point of the $\C^*$-action is of codimension
at least 2 in $T$. By the Hartogs' theorem $f\circ \nu$ extends regularly to $T_N$  which concludes (3).

\end{proof}

\begin{remark}\label{fabrizio.1}  Consider the rational map $\kappa : T  \to Z$ given by $t \mapsto ({\frac{1}{f(t)}}, 1). t$.  It is regular on $T\setminus S$ and if $t\in T\setminus S$ and $z \in Z$ are such that
$t=(\lambda ,1).z$ then $\kappa (t) =z$. In particular $\kappa$ sends $\T$-orbits from $T$ into $\C^*$-orbits of $Z$. Furthermore, morphism $\kappa_N=\kappa \circ \nu : T_N \to Z$ is regular by the same reason as function $f \circ \nu$ is.

\end{remark}

\begin{lemma}
\label{codim2} Let $E_i= \{ t=(a_1,a_2,b_1,b_2, w, v )\in \bar T | b_i=0\}$ and
$\bar T^b$ coincide with $T \cup ((f)_0 \setminus E_2) \cup ((f)_{\infty} \setminus E_1)$.
Suppose that $\bar T_N^b $ is a normalization of $ \bar T^b$. Then there is a regular extension
of $\kappa_N : T_N \to Z$ to a morphism $\bar \kappa_N^b : \bar T_N^b \to Z$.
\end{lemma}

\begin{proof}

Since the set $(f)_0 \cap (f)_{\infty}$ is of codimension 2 in $\bar T$, the
Hartogs' theorem implies that
 it suffices to prove the regularity of $\bar \kappa_N^b$ on the normalization of $\bar T^b \setminus ((f)_0\cap (f)_{\infty})$.  Furthermore, by the Riemann extension theorem it is enough to construct a continuous extension of $\kappa$ from $T\setminus S$ to $\bar T^b \setminus (S \cup ((f)_0\cap (f)_\infty ))$.


By Lemma \ref{function} (2) we need to consider this extension, say, at
 $t=(a_1,a_2,b_1,b_2,w,v) \in (f)_0 \setminus (f)_{\infty}$. Let
 $t_n \to t$ as $n \to \infty$ where $$
 t_n=(f(t_n)a_1^n,{\frac {1} {
f(t_n)}}a_2^n,b_1^n,b_2^n, f(t_n).v_n,v_n)\in T$$  with
$a_1^nb_2^n-a_2^nb_1^n=1$ and $f(t_n) \to 0$.
Perturbing, if necessary, this sequence $\{ t_n \}$ we can suppose every $t_n\notin S$, i.e.
$\kappa (t_n)=(a_1^n,a_2^n,b_1^n,b_2^n,v_n,v_n)$.
 Note that $\lim v_n =v$,  $b_k=\lim b_k^n, \, k=1,2$ and $a_2^n\to 0$ since $a_2$ is finite. Hence $1=a_1^nb_2^n-a_2^nb_1^n\approx a_1^nb_2$ and $a_1^n \to 1/b_2$ as $n \to\infty$.  Now we get a continuous extension of $\kappa$ by putting $\kappa (t)=(1/b_2, 0,b_1,b_2,v,v)$. This yields the desired conclusion.

\end{proof}

\begin{remark}\label{fabrizio.2} If we use the group $(1\times \C^*)$ instead of the group $(\C^* \times 1)$ from Lemma \ref{easy} (ii) in our construction this would lead to the replacement of $f$ by $f^{-1}$. Furthermore for the variety $\bar T^a= T\cup ((f)_0\setminus \{ a_1=0 \}) \cup ((f)_{\infty} \setminus \{ a_2=0 \})$ we obtain a morphism $\bar \kappa^a_N : \bar T_N^a \to  Z$ similar to $\bar \kappa_N^b$.

\end{remark}

The next fact is intuitively obvious but requires some work.

\begin{lemma}\label{fabrizio.3} The complement $\bar T^0$ of $\bar T^a \cup \bar T^b$  in $\bar T$(which is
$\bar T^0=(\bar T \setminus T) \cap \bigcup_{i\ne j} \{ a_i =b_j= 0 \} $) has codimension at least 2.

\end{lemma}

\begin{proof}
Let $t_n \to t=(a_1,a_2,b_1,b_2,w,v)$ be as in the proof of Lemma \ref{codim2}.
Since for a general point of the slice
$V$ the isotropy group is finite after perturbation we can suppose
that each $v_n$ is contained in a non-constant $\C^*$-orbit
$O_n\subset V$.
Treat $v_n$ and $f(t_n).v_n$ as numbers in $\C^*
\simeq O_n$ such that $f(t_n).v_n=f(t_n)v_n$. Let $|v_n|$ and
$|f(t_n).v_n|$ be their absolute values. Then one has the annulus
$A_n =\{ |f(t_n).v_n| < \zeta <|v_n| \} \subset O_n$, i.e. $\zeta
= \eta v_n$ where $|f(t_n)| <|\eta | <1$ for each $\zeta \in A_n$.
By Lemma \ref{easy} (iii) $\pi (v) =\pi (w)$ but by Lemma
\ref{function} (3) the $\C^*$-orbit $O(v)$ and $O(w)$ are
different unless $w=v$ is a fixed point of the $\C^*$-action. In
any case, by Proposition \ref{redu} (2) the closures of these
orbits meet at a fixed point $\bar v$ of the $\C^*$-action.

Consider a compact neighborhood $W=\{ u\in V | \phi (u) \leq 1 \}$
of $\bar v$ in $V$ where $\phi$ is a plurisubharmonic function on
$V$ that vanishes at $\bar v$ only. Note that the sequence $\{
(\lambda , \mu). t_n \} $ is convergent to $ (\lambda a_1, a_2/
\lambda ,\mu b_1,b_2/\mu ,\lambda . w, \mu .v)$. In particular,
replacing $\{ t_n \} $ by $\{ (\lambda , \mu). t_n \} $ with
appropriate $\lambda$ and $\mu$ we can suppose that the boundary
$\partial A_n$ of any annulus $A_n$ is contained in $W$ for
sufficiently large $n$. By the
maximum principle $\bar A_n \subset W$.
The limit $A=\lim_{n \to \infty} \bar A_n$ is a compact subset of
$W$ that contains both $v$ and $w$, and also all points $\eta . v$
with $0< |\eta |<1$ (since $|f(t_n)| \to 0$).
Unless $O(v)=\bar v$ only one of the closures of sets $\{ \eta . v
| \, 0< |\eta |<1 \} $ or $\{ \eta . v | \, |\eta |>1 \} $ in $V$
is compact and contains the fixed point $\bar v$ (indeed,
otherwise the closure of $O(v)$ is a complete curve in the affine
variety $V$). The argument before shows that it is the first one.

That is, $\mu . v \to \bar v$ when $\mu \to 0 $. Similarly,
$\lambda . w \to \bar v$ when $\lambda \to \infty$. It is not difficult to check now
that the dimension of the set of such pairs $(w,v)$ is at most $\dim V$.

Consider the set $(\bar T \setminus T) \cap \{ a_1=b_2=0 \}$. It consists of points
 $t=(0,a_2,b_1,0,w,v)$ and, therefore, its dimension, is at most $\dim V+2$.
 Thus it has codimension at least 2 in $\bar T$ whose dimension is $\dim V+4$.
This yields the desired conclusion.

\end{proof}

The next technical fact may be somewhere in the literature, but
unfortunately we did not find a reference. 

\begin{proposition}\label{fabrizio.26} Let a reductive group $G$ act on an affine algebraic
variety $X$ and $ \pi : X \to Q:=X//G$ be the quotient morphism such that
one of closed $G$-orbits $O$ is contained in the smooth part of $X$. Suppose that $\nu : X_N \to X$
and $\mu : Q_N \to Q$ are normalization morphisms, i.e. $\pi \circ \nu = \pi_N \circ \mu$ for some
morphism $\pi_N : X_N \to Q_N$. Then $Q_N \simeq X_N//G$ for the induced $G$-action on $X_N$ and
$\pi_N$ is the quotient morphism.

\end{proposition}

\begin{proof} Let $\psi :  X_N \to R$ be the quotient morphism.
By the universal property of quotient morphims
$\pi_N = \varphi \circ \psi$ where $\varphi : R \to Q_N$ is a morphism. It suffices to show
that $\varphi$ is an isomorphism.
The points of $Q$ (resp. $R$) are nothing but the closed $G$-orbits in $X$ (resp. $X_N$) by
Proposition \ref{redu}, and above each closed orbit in $X$ we have only a finite number of closed
orbits in $X_N$ because $\nu$ is finite. Hence $\mu \circ \varphi : R \to Q$ and,
therefore, $\varphi : R \to Q_N$ are at least quasi-finite.
There is only one closed orbit $O_N$ in $X_N$ above orbit $O\subset {\rm reg} \, X$. Thus $\varphi$ is injective
in a neighborhood of $\psi (O_N)$. That is, $\varphi$ is
birational and by the Zariski Main theorem it is an embedding.

It remains to show that
$\varphi$ is proper.  Recall that $G$ is a complexification of a its compact subgroup $G^{\R}$
and there is a so-called Kempf-Ness real algebraic subvariety $X^{\R}$ of $X$
such that the restriction $\pi |_{X^{\R}}$ is nothing but the standard quotient map $X^{\R}
\to X^{\R}/G^{R}=Q$ which is automatically proper (e.g., see \cite{Sch}). Set $X_N^{\R}= \nu^{-1} (
X^\R )$. Then the restriction of $\pi \circ \nu$ to $X_N^\R$ is proper being the composition
of two proper maps. On the other hand the restriction of
$\mu \circ \pi_N = \pi \circ \nu$ to $X_N^\R$ is proper
only when morphism $\varphi$, through which it factors, is proper which concludes the proof.

\end{proof}

\begin{proposition}
\label{finite} Morphism $\rho^{\prime}: Y^{\prime}\to Z^{\prime}$ from
diagram \ref{YZ} is finite birational.
\end{proposition}

\begin{proof} Morphism $\rho'$ factors through $\rho_N' : Y' \to Z_N'$ where $\mu : Z_N' \to Z'$ is a normalization
of $Z'$ and the statement of the proposition is equivalent to the fact that $\rho_N'$ is an isomorphism.
Set $Z'(b)= q(\bar T^b)$ and $Z'(a)=q(\bar T^a)$. Note that $Z'\setminus (Z'(a) \cup Z'(b))$ is in the $q$-image of the $\T$-invariant set $\bar T^0$ from Lemma \ref{fabrizio.3}. Hence  $Z'\setminus (Z'(b) \cup Z'(a))$ is of codimension 2 in $Z'$ and  by the Hartogs' theorem it suffices to prove that $\rho_N'$ is invertible over $Z(b)'$ (resp. $Z'(a))$.

By Remark \ref{fabrizio.1} $\bar \kappa_N^b$ sends each
orbit of the induced $\T$-action on $\bar T_N^b$ onto a
$\C^*$-orbit in $Z$. Thus the composition of $\bar \kappa_N^b$ with $p : Z
\simeq Y \to Y'$ is constant on $\T$-orbits and by the universal
property of quotient spaces it must factor through the quotient
morphism $q_N^b : \bar T_N^b \to Q$. By Proposition \ref{fabrizio.26}
$Q= Z_N'(b)$ where $Z_N'(b)=\mu^{-1}(Z'(b))$.
That is, $p \circ \bar \kappa_N^b = \tau^b
\circ q_{N}^b$ where $\tau^b : Z_N'(b) \to Y'$.
Our construction implies that $\tau^b$ is the inverse of
$\rho_{N}'$ over $Z_N'(b)$. Hence $\rho_{N}'$ is invertible over $Z_N'(b)$
which concludes the proof.

\end{proof}

\subsection{Proof of Theorems \ref{goal} and \ref{MAIN}.}

Let $G=SL_2$ act algebraically on $X$ as in Theorem \ref{goal} and $V$ be
the slice of this action at point $x \in X$ so that there is an \'etale
morphism $G\times_{G_x} V \to U$ as in Theorem \ref{Luna}. By Lemmas \ref
{local} and \ref{product} for validity of Theorem \ref{goal} it suffices to
prove semi-compatibility of vector fields $\delta_1$ and $\delta_2$ on ${%
\ensuremath{\mathcal{Y}}} =G \times_{G_x} V$, which was already done in the
case of a finite isotropy group $G_x$ (see Lemma \ref{product} (2)).
Consider the quotient morphisms $\varrho_i : {\ensuremath{\mathcal{Y}}} \to {%
\ensuremath{\mathcal{Y}}}_i : ={\ensuremath{\mathcal{Y}}} //H_i$ where $H_i,
\, i=1,2$ are as in Example \ref{example.01}. Set $\varrho = (\varrho_1,
\varrho_2) : {\ensuremath{\mathcal{Y}}} \to \overline {\varrho ({%
\ensuremath{\mathcal{Y}}} )} \subset {\ensuremath{\mathcal{Y}}}_1 \times {%
\ensuremath{\mathcal{Y}}}_2$. By Proposition \ref{reformulation} Theorem \ref
{goal} is true if $\varrho$ is finite birational. If $G_x=%
\ensuremath{\mathbb{C}}^*$ then $\varrho : {\ensuremath{\mathcal{Y}}} \to {%
\ensuremath{\mathcal{Y}}}_1 \times {\ensuremath{\mathcal{Y}}}_2$ is nothing
but morphism $\rho^{\prime}: Y^{\prime}\to Y_1^{\prime}\times Y_2^{\prime}$
from Proposition \ref{finite}, i.e. we are done in this case as well. By
Lemma \ref{isosl2} the only remaining case is when $G_x$ is an extension of $%
\ensuremath{\mathbb{C}}^*$ by $\ensuremath{\mathbb{Z}}_2$. Then one has a $%
\ensuremath{\mathbb{Z}}_2$-action on $Y^{\prime}$ such that it is
commutative with $H_i$-actions on $Y^{\prime}$ and ${\ensuremath{\mathcal{Y}}%
} =Y^{\prime}//\ensuremath{\mathbb{Z}}_2$. Since vector fields $\delta_1$
and $\delta_2$ are semi-compatible on $Y^{\prime}$ by Propositions \ref
{finite} and \ref{reformulation}, they generate also semi-compatible vector
fields on ${\ensuremath{\mathcal{Y}}}$ by Lemma \ref{finereformulation}.
This concludes Theorem \ref{goal} and, therefore, Theorem \ref{MAIN}.
\hspace{4in} $\square$

\begin{remark}
\label{dynk} (1) Consider ${\ensuremath{\mathcal{Y}}} =G \times_{G_x} V$ in
the case when $G=G_x =SL_2$, i.e. the $SL_2$-action has a fixed point. It is
not difficult to show that morphism $\varrho = (\varrho_1, \varrho_2) : {%
\ensuremath{\mathcal{Y}}} \to \overline {\varrho ({\ensuremath{\mathcal{Y}}}
)} \subset {\ensuremath{\mathcal{Y}}}_1 \times {\ensuremath{\mathcal{Y}}}_2$
as in the proof before is not quasi-finite. In particular, $\delta_1$ and $%
\delta_2 $ are not compatible. However, we do not know if the condition
about the absence of fixed points is essential for Theorem \ref{MAIN}. In
examples we know the presence of fixed points is not an obstacle for the
algebraic density property. Say, for $\ensuremath{\mathbb{C}}^n$ with $2\leq
n \leq 4$ any algebraic $SL_2$-action is a representation in a suitable
polynomial coordinate system (see, \cite{Po}) and, therefore, has a fixed
point; but the validity of the algebraic density property is a consequence
the Anders\' en-Lempert work.

(2) The simplest case of a degenerate $SL_2$-action is presented by the
homogeneous space $SL_2/\ensuremath{\mathbb{C}}^*$ where $%
\ensuremath{\mathbb{C}}^*$ is the diagonal subgroup. Let
\begin{equation*}
A = \left(
\begin{array}{ccc}
a_1 & a_2 &  \\
b_1 & b_2 &
\end{array}
\right)
\end{equation*}
be a general element of $SL_2$.
Then the ring of invariants of the $\ensuremath{\mathbb{C}}^*$-action is
generated by $u=a_1a_2$, $v=b_1b_2$ , and $z=a_2b_1+1/2$ (since $%
a_1b_2=1+a_2b_1=1/2+z$). Hence $SL_2/\ensuremath{\mathbb{C}}^*$ is
isomorphic to a hypersurface $S$ in $\ensuremath{\mathbb{C}}^3_{u,v,z}$
given by the equation $uv=z^2-1/4 $. In particular, it has the algebraic
density property by \cite{KK1}.

(3) However, the situation is more complicated if we consider the normalizer
$T$ of the diagonal $\ensuremath{\mathbb{C}}^*$-subgroup of $SL_2$ (i.e. $T$
is an extension of $\ensuremath{\mathbb{C}}^*$ by $\ensuremath{\mathbb{Z}}_2$%
). Then ${\ensuremath{\mathcal{P}}} =SL_2/T$ is isomorphic to $S/%
\ensuremath{\mathbb{Z}}_2$ where the $\ensuremath{\mathbb{Z}}_2$-action is
given by $(u,v,z) \to (-u,-v,-z)$. It can be shown that this surface ${%
\ensuremath{\mathcal{P}}}$ is the only $\ensuremath{\mathbb{Q}}$-homology
plane which is simultaneously a Danilov-Gizatullin surface (i.e. it has a
trivial Makar-Limanov invariant (see \cite{FKZ})), and its fundamental group
is $\ensuremath{\mathbb{Z}}_2$. We doubt that ${\ensuremath{\mathcal{P}}}$
has algebraic density property.
\end{remark}

\section{Applications}

Theorem \ref{MAIN} is applicable to a wide class of homogeneous spaces.
%
Let us start with the following observation: given a reductive
subgroup $R$ of a linear algebraic group $G$ any $SL_2$-subgroup
$\Gamma <G$ yields a natural $\Gamma$-action on $G /R$.
Furthermore, for each point $aR \in G/R$ its isotropy subgroup
under this action is isomorphic to $\Gamma  \cap a
Ra^{-1}$. In particular, the action has no fixed points if
 $a^{-1}\Gamma a$ is not contained in $R$ for any $a
\in G$ and it is non-degenerate if $\Gamma^a:=a^{-1}\Gamma a \cap R \simeq
\Gamma \cap a R a^{-1}$ is finite for some
$a \in R$. Thus Theorem \ref {MAIN} implies the following.

\begin{proposition}
\label{prop.10} Let $G$ be an algebraic group and $\Gamma_1,
\ldots , \Gamma_k$ be its $SL_2$-subgroups such that at some $x
\in G$ the set $\{ \delta_2^i (x) \}$ is a generating one (where
$(\delta_1^i, \delta_2^i)$ is the corresponding pair of locally
nilpotent vector fields on $G$ generated by the natural
$\Gamma_i$-action). Suppose that for each $i=1, \ldots ,k$ and any
$a \in G$ the group $\Gamma_i^a:=a^{-1}\Gamma_ia$ is not isomorphic to $ \Gamma_i$,
and furthermore $\Gamma_i^a$ is finite for some $a$. Then $G / R$
has the algebraic density property.
\end{proposition}

Note that for a simple Lie group $G$ a generating set at any $x \in G$
consists of one nonzero vector since the adjoint representation is
irreducible. Therefore, in this case the algebraic density property is a
consequence of the following.

\begin{theorem}
\label{dynkin} Let $G$ be a simple Lie group with Lie algebra
different from $\sgoth \lgoth_2$ and $R$ be its proper reductive
subgroup. Then there exists an $SL_2$-subgroup $\Gamma$ in $G$
such that $\Gamma^a $ is not isomorphic to $ \Gamma$ for any $a \in G$, and, furthermore, $%
\Gamma^a$ is finite for some $a \in G$.
\end{theorem}

Surprisingly enough the proof of this Theorem (at least in our presentation)
requires some serious facts from Lie group theory and we shall postpone it
till the Appendix.

\begin{corollary}
\label{cor.40} Let $X=G/R$ be an affine homogeneous space of a semi-simple
Lie group $G$. Suppose that $X$ is different from a $\ensuremath{\mathbb{Q}}$%
-homology plane with $\ensuremath{\mathbb{Z}}_2$ as a fundamental group.
Then $X$ is equipped with $N$ 
pairs $\{ \delta_1^k,\delta^k_2 \}_{k=1}^N$ of compatible derivations such
that the collection $\{\delta^k_2 (x_0) \}_{k=1}^N \subset T_{x_0}X$ is a
generating set at some point $x_0\in X$. In particular, $X$ has the
algebraic density property by Theorem \ref{density}.
\end{corollary}

\begin{proof} Note that $R$ is reductive by Proposition \ref{matsu}
(Matsushima's theorem). Then $X$ is isomorphic to a quotient of
form $G/R$ where $G=G_1 \oplus \ldots \oplus G_N$, each $G_i$ is a
simple Lie group and $R$ is not necessarily connected. However, we
can suppose that $R$ is connected by virtue of Proposition
\ref{finereformulation}. Consider the projection homomorphism
$\pi_k : G \to G_k$ and $R_k=\pi_k (R)$ which is reductive being
the image of a reductive group. If $N$ is minimal possible then
$R_k \ne G_k$ for every $k$. Indeed, if say $R_N =G_N$ then
$X=(G_1 \oplus \ldots \oplus G_{N-1})/\tilde R$ where $\tilde
R=\Ker \pi_N$ which contradicts minimality.

Assume first that none of $G_i$'s is isomorphic to $SL_2$. By
Theorem \ref{dynkin}  one can choose an $SL_2$-subgroup
$\Gamma_k <G_k$ such that the natural $\Gamma_k$-action on
$G_k/\pi_k ( R)$ and, therefore, on $G/R$ is fixed point free.
We can also assume that each $\Gamma_k$-action is non-degenerate.
Denote by $\delta_1^k$ and $\delta_2^k$ the
corresponding pair of locally nilpotent derivations for the
$\Gamma_k$-action. Since the adjoint representation is irreducible
for a simple Lie group, $\{\delta^k_2 (e) \}_{k=1}^N$ is a
generating set of the tangent space $T_eG$ at $e=e_1 \oplus \ldots
\oplus e_N \in G$, where $e_k$ is a unit of $G_k$. Consider $X=G/R$
as the set of left cosets, i.e. $X$ is the quotient of $G$ with
respect to the action generated by multiplication by elements of
$R$ from the right. Hence this action commutes with multiplication
by elements of $\Gamma_k$ from the left, and, therefore, it
commutes with any field $\delta_k^i$. Pushing the actions of
$\Gamma_k$'s to $X$ we get fixed point free non-degenerate
$SL_2$-actions on $X$ and the desired conclusion in this case
follows from  by Theorem \ref{MAIN}.

In the case when some of $G_k$'s are isomorphic to $SL_2$ we
cannot assume that  each $\Gamma_k$-action is non-degenerate, but now
$N \geq 2$ and the $\Gamma_k$-actions are still fixed point free.
Consider an isomorphism $\varphi_k : \Gamma_k \to \Gamma_1$. Then
we have an $SL_2$-group $\Gamma^{\varphi_k}= \{( \varphi_k (\gamma
), \gamma ) | \gamma \in \Gamma_k \} < \Gamma_1 \times \Gamma_k$
acting naturally on $G_1 \times G_k$ and, therefore, on $G$. This
isomorphism $\varphi_k$ can be chosen so that the
$\Gamma^{\varphi_k}$-action is non-degenerate. Indeed, if, say a
$\C^*$-subgroup $L <\Gamma_k$ acts on $G_k$ trivially choose
$\varphi_k$ so that $\varphi_k (L)$ acts nontrivially on $G_1$
which makes the action non-degenerate. In particular, by Theorem
\ref{goal} we get pairs of compatible locally nilpotent
derivations $\tilde\delta_1^{\varphi_k}$ and $\tilde
\delta_2^{\varphi_k}$ corresponding to such actions. Set $G' =G_2
\oplus \ldots \oplus G_N$ and $e' =e_2 \oplus \ldots \oplus e_N
\in G'$. Since the adjoint representation is irreducible for a
simple Lie group the orbit of the set $\{\tilde
\delta^{\varphi_k}_2 (e) \}_{k=2}^N$ under conjugations generates
a subspace of $S$ of $T_eG$ such that the restriction of the
natural projection $T_e G \to T_{e'}G'$ to $S$ is surjective. In
order to enlarge $\{\tilde \delta^{\varphi_k}_2 (e) \}_{k=2}^N$ to
a generating subset of $T_eG$ consider an isomorphism $\psi_2 :
\Gamma_2 \to \Gamma_1$ different from $\varphi_2$ and such that
the $\Gamma^{\psi_2}$-action is non-degenerate. Denote the
corresponding compatible locally nilpotent derivations by $\tilde
\delta_1^{\psi_2}$ and $\tilde \delta_2^{\psi_2}$ on $G_1 \oplus
G_2$ (and also by abusing notation on $G$). Note that the vectors
$\tilde \delta^{\varphi_2}_2 (e_1\oplus e_2)$ and $\tilde
\delta^{\psi_2}_2 (e_1\oplus e_2)$ can be assumed different with
an appropriate choice of $\psi_2$. Hence these two vectors form a
generating subset of $T_{e_1\oplus e_2} G_1 \oplus G_2$. Taking
into consideration the remark about $S$ we see that $\{\tilde
\delta^{\varphi_k}_2 (e) \}_{k=2}^N \cup \{ \tilde \delta^{\psi_2}_2
(e) \}$ is a generating subset of $T_e G$. Now pushing these
$SL_2$-actions to $X$ we get the desired conclusion.

\end{proof}

\begin{theorem}\label{thm.50}
Let $G$ be a linear algebraic group and $R$ be its proper
reductive subgroup such that the homogeneous space $G/R$ is
different from $\ensuremath{\mathbb{C}}_+$, a torus, or the
$\ensuremath{\mathbb{Q}}$ -homology plane with fundamental group
$\ensuremath{\mathbb{Z}}_2$ (i.e. the surface
${\ensuremath{\mathcal{P}}}$ from Remark \ref{dynk} (3)). Then
$G/R$ has the algebraic density property.
\end{theorem}

\begin{proof} Since all components of $G/R$ are isomorphic as varieties we can
suppose that $G$ is connected. Furthermore, by Corollary
\ref{cor.40} and Remark \ref{dynk} (2) we are done with a
semi-simple $G$.

Let us consider first the case of a reductive but not semi-simple
$G$. Then the center $Z\simeq (\C^*)^n$ of $G$ is nontrivial. Let
$S$ be the semi-simple part of $G$. Assume for the time being that
$G$ is isomorphic as group to the direct product $S\times Z$ and
consider the natural projection $\tau : G \to Z$. Set $Z'=\tau
(R)=R/R'$ where $R' =R \cap S$. Since we are going to work with
compatible vector fields we can suppose that $R$ is connected by
virtue of Lemma \ref{finereformulation}. Then $Z'$ is a subtorus
of $Z$ and also $R'$ is reductive by Proposition \ref{matsu}.
Hence $G/R=(G/R')/Z'$ and $G/R'=S/R'\times Z$. Note that there is
a subtorus $Z''$ of $Z$ such that $Z'' \simeq Z/Z'$ and $Z' \cdot
Z''=Z$. (Indeed, $Z' \simeq (\C^*)^k$ generates a sublattice $L
\simeq \Z^k$ of homomorphisms from $\C^*$ into $Z'$ of the similar
lattice $\Z^n$ of $Z\simeq (\C^*)^n$ such that the quotient
$\Z^n/L$ has no torsion, i.e. it is isomorphic to $\Z^{n-k}$.
Since any short exact sequence of free $\Z$-modules splits we have
a $\Z$-submodule $K \simeq \Z^{n-k}$ in $\Z^n$ such that $K+L
=\Z^n$. This lattice $K$ yields a desired subtorus $Z''\simeq
(\C^*)^{n-k}$.)
Hence $G/R$ is isomorphic to $\varrho^{-1} (Z'') \simeq S/R'
\times Z''$ where $\varrho : G/R' \to Z$ is the natural
projection.
Note that both factors are nontrivial since otherwise $G/R$ is
either a torus or we are in the semi-simple case again.
Thus $X$ has the algebraic density property by Proposition
\ref{prod.11} with $S/R'$ playing the role of $X_1$ and $Z''$ of
$X_2$. In particular, we have a finite set of pairs of compatible
vector fields $\{ \delta_1^k , \delta_2^k \}$ as in Theorem
\ref{density}. Furthermore, one can suppose that the fields
$\delta_1^k$ correspond to one parameter subgroups of $S$
isomorphic to $\C_+$ and $\delta_2^k$ to one parameter subgroups
of $Z$ isomorphic to $\C^*$. In the general case $G/R$ is the
factor of $X$ with respect to the natural action of a finite
(central) normal subgroup $F<G$. Since $F$ is central the fields $
\delta_1^k$, $\delta_2^k $ induce completely integrable vector
fields $\tilde \delta_1^k$, $\tilde \delta_2^k $ on $G/R$ while
$\tilde \delta_2^k (x_0)$ is a generating set for some $x_0 \in
G/R$. By Lemma \ref{finereformulation} the pairs $\{ \tilde
\delta_1^k, \tilde \delta_2^k \}$ are compatible and the density
property for $G/R$ follows again from Theorem \ref{density}.

In the case of a general linear algebraic group $G$ different from
a reductive group, $\C^n$, or a torus $(\C^*)^n$ consider the
nontrivial unipotent radical $\cR_u$ of $G$. It is automatically
an algebraic subgroup of $G$ (\cite{Ch}, p. 183). By Mostow's
theorem \cite{Mo} (see also \cite{Ch}, p. 181) $G$ contains a
(Levi) maximal closed reductive algebraic subgroup $G_0$
such that $G$ is the semi-direct product of $G_0$ and $\cR_u$,
i.e. $G$ is isomorphic as affine variety to the product $\cR_u
\times G_0$. Furthermore, any other maximal reductive subgroup is
conjugated to $G_0$.  Hence, replacing $G_0$ by its conjugate, we
can suppose that $R$ is contained in $G_0$. Therefore $G/R$ is
isomorphic as an affine algebraic variety to the $G_0/R \times
\cR_u$ and we are done now by Proposition \ref{prod.11} with
$\cR_u$ playing the role of $X_1$ and $G_0/R$ of $X_2$.
\end{proof}

\begin{remark}
(1) The algebraic density property implies, in particular, that
the Lie algebra generated by completely integrable algebraic (and,
therefore, holomorphic) vector fields is infinite-dimensional,
i.e. this is true for homogeneous spaces from Theorem
\ref{thm.50}. For Stein manifolds of dimension at least two that
are homogeneous spaces of holomorphic actions of a connected
complex Lie groups the infinite dimensionality of such algebras
was also established by Huckleberry and Isaev \cite{HI}.

(2) Note that as in \cite{KK2} we proved actually a stronger fact
for a homogeneous space $X=G/R$ from Theorem \ref{thm.50}. Namely,
it follows from the construction that the Lie algebra generated by
vector fields of form $f \sigma$, where $\sigma$ is either locally
nilpotent or semi-simple and $f \in \Ker \sigma$ for semi-simple
$\sigma$ and $\deg_{\sigma} f \leq 1$ in the locally nilpotent
case, coincides with $\AVF (X)$.
\end{remark}

\section{Appendix: The proof of Theorem \ref{dynkin}.}

Let us start with the following technical fact.

\begin{proposition}
\label{prop.30} Let $R$ be a semi-simple subgroup of a semi-simple group $G$%
. Suppose that the number of orbits of nilpotent elements in the Lie algebra
${\ensuremath{\mathfrak{r}}}$ of $R$ under the adjoint action is less than
the number of orbits of nilpotent elements in the Lie algebra of $G$ under
the adjoint action. Then $G$ contains an $SL_2$-subgroup $\Gamma$ such that $%
\Gamma^g :=g^{-1}\Gamma g \cap R$ is different from $g^{-1}\Gamma g$ for any
$g \in G$.
\end{proposition}

\begin{proof} By the Jacobson-Morozov theorem (\cite{Bo}, Chap. 8.11.2, Proposition 2
and Corollary) for any semi-simple group $G$ there is a bijection
between the set of $G$-conjugacy classes of
$\sgoth\lgoth_2$-triples and the set of $G$-conjugacy classes of
nonzero nilpotent elements from $G$ which implies the desired
conclusion.

\end{proof}

In order to exploit Proposition \ref{prop.30} we need to remind some
terminology and results from \cite{Bo}.

\begin{definition}
\label{principal} (1) Recall that a semi-simple element $h$ of a Lie algebra
is regular, if the kernel of its adjoint action is a Cartan subalgebra. An ${%
\ensuremath{\mathfrak{s}}}{\ensuremath{\mathfrak{l}}}_2$-subalgebra
of the Lie algebra
${\ensuremath{\mathfrak{g}}}$ of a
semi-simple group $G$ is called principal if in its triple of
standard
generators the semi-simple element $h$ is regular and the adjoint action of $%
h$ has even eigenvalues (see Definition 3 in \cite{Bo} Chapter 8.11.4). The
subgroup generated by this subalgebra is called a principal $SL_2$-subgroup
of $G$. As an example of such a principal subgroup one can consider an $SL_2$%
-subgroup of $SL_n$ that acts irreducibly on the natural representation
space $\ensuremath{\mathbb{C}}^n$. In general, principal ${%
\ensuremath{\mathfrak{s}}}{\ensuremath{\mathfrak{l}}}_2$-subalgebras exist
in any semi-simple Lie algebra ${\ensuremath{\mathfrak{g}}}$ (see
Proposition 8 in \cite{Bo} Chapter 8.11.4). Any two principal $SL_2$%
-subgroups are conjugated (see Proposition 6 in \cite{Bo} Chapter 8.11.3 and
Proposition 9 in \cite{Bo} Chapter 8.11.4).

(2) A connected closed subgroup $P$ of $G$ is called principal if it
contains a principal $SL_2$-subgroup\footnote{%
A definition of a principal subgroup in \cite{Bo} is different (see Exercise
18 for Chapter 9.4) but it coincides with this one in the case of a complex
Lie group (see Exercise 21c for Chapter 9.4).}. A rank of $P$ is the rank of
the maximal torus it contains. If this rank is 1 then $P$ coincides with its
principal $SL_2$-subgroup (see Exercise 21 for Chapter 9.4).

\end{definition}

\begin{proposition}
\label{prin.10} Let $R$ be a proper reductive subgroup of a simple group $G$ different
from $SL_2$ or $PSl_2$.
Then there exists an $SL_2$-subgroup $\Gamma$ of $G$ such that $%
\Gamma^g :=g^{-1}\Gamma g \cap R$ is different from $g^{-1}\Gamma g$ for any
$g \in G$.
\end{proposition}

\begin{proof} If $R$ is not principal it cannot contain a principal
$SL_2$-subgroup and we are done. Thus it suffices to consider the
case of principal subgroup $R$ only.

Suppose first that $R$ is of rank 1. If $R$ contains $g^{-1}
\Gamma g$ it must coincide with this subgroup by the dimension
argument. Hence it suffices to choose non-principal $\Gamma$ to
see the validity of the Proposition in this case.

Suppose now that $R$ is of rank at least 2. Then there are the
following possibilities (\cite{Bo}, Exercises 20c-e for Chapter
9.4):

(1) $R$ is of type $B_2$ and $G$ is of type $A_3$ or $A_4$;

(2) $R$ is of type $G_2$ and $G$ is of type $B_3,D_4,$ or $A_6$;

(3) $G$ is of type $A_{2l}$ with $l \geq 3$ and $R$ is of type
$B_l$;

(4) $G$ is of type $A_{2l-1}$ with $l \geq 3$ and $R$ is of type
$C_l$;

(5) $G$ is of type $D_{l}$ with $l \geq 4$ and $R$ is of type
$B_{l-1}$;

(6) $G$ is of type $E_6$ and $R$ is of type $F_4$.

In order to apply Proposition \ref{prop.30} to these cases  we
need the Dynkin classification of nilpotent orbits (with
Elkington's corrections) as described in the Bala-Carter paper
(\cite{BaCa2} page 6-7).

By this classification the number $a_n$ of such orbits in a simple
Lie algebra of type $A_n$ coincides with the number of partitions
$\lambda$ of $n+1$, i.e. $\lambda =(\lambda_1, \ldots ,\lambda_k)$
with natural $\lambda_i$ such that $|\lambda | = \lambda_1 \ldots
+ \lambda_k=n+1$.

For a simple Lie algebra of type $B_m$ the number $b_m$ of
nilpotent orbits coincides with number of partitions $\lambda$ and
$\mu$ such that $2|\lambda|+|\mu|=2m+1$ where $\mu$ is a partition
with distinct odd parts.

For a simple Lie algebra of type $C_m$ the number $c_m$ of
nilpotent orbits coincides with number of partitions $\lambda$ and
$\mu$ such that $|\lambda|+|\mu|=m$ where $\mu$ is a partition
with distinct parts.

For a simple Lie algebra of type $D_m$ the number $d_m$ of
nilpotent orbits coincides with number of partitions $\lambda$ and
$\mu$ such that $2|\lambda|+|\mu|=2m$ where $\mu$ is a partition
with distinct odd parts.

The number of nilpotent orbits of algebras of type $G_2,
F_4,E_6,E_7,E_8$ is 5,16, 21,45, and 70 respectively.

Now one has $a_4>a_3=5 >b_2=4$ which settles case (1) by
Proposition \ref{prop.30}. Then $b_3,d_4,a_6
> 5$ which settles case (2). Similarly, $a_{2l} >b_l$ for $l \geq 3$,
$a_{2l-1}>c_l$ for $l \geq 3$, $d_l >b_{l-1}$ for $l \geq 4$, and
$21>16$ which settles cases (3)-(6) and concludes the proof.

\end{proof}

\begin{remark}
In fact, the statement of the above Proposition is true for any proper maximal
subgroup $R$ of $G $. This can be deduced from Dynkin's classification
of maximal subalgebras in semisimple Lie algebras. We outline the argument
below.

Let us consider a maximal subalgebra $\frak{r}$ in $\frak{g}$, where $\frak{g%
}$ is a simple Lie algebra. If $\frak{r}$ is regular (i.e., if its
normalizer contains some Cartan subalgebra in $\frak{g}$), then
$\frak{r}$ does not contain \ any principal $\frak{sl}_{2}$-triple
\cite[Section 6.2.4]{OV}. Thus we may assume that $\frak{r}$ $\
$is non-regular.

If $\frak{g}$ is exceptional, the list of such $\frak{r}$ is given in
\cite[Theorems 6.3.4, 6.3.5]{OV}. All of them are semisimple, and we will
only consider simple subalgebras (otherwise, $\frak{r}$ once again does not
contain any principal $\frak{sl}_{2}$ 's). The list of simple maximal
non-regular subalgebras of rank $\geq 2$ in exceptional Lie algebras is
short: $B_{2}$ in $E_{8}$, $A_{2}$ in $E_{7}$ and \thinspace $A_{2}$, $G_{2}$%
, $C_{4}$, $F_{4}$ in $E_{6}$. In all these cases Proposition \ref{prop.30}
applies.

It remains to consider non-regular maximal subalgeras $\frak{r}$ of
classical Lie algebras. Any such $\frak{r}$ is simple, and an embedding of $%
\frak{r}$ in $\frak{g}$ is defined by a nontrivial linear irreducible
representation $\phi :\frak{r}\rightarrow \frak{sl}(V)$. Let $n=\dim V$ and $%
m=\left[ \frac{n}{2}\right] $. If the module $V$ is not self-dual,
$\frak{r} $ is a maximal subalgebra in $\frak{g}=A_{n-1}.$ If $V$
is self-dual and endowed with a skew-symmetric invariant form,
$\frak{r}$ is a maximal subalgebra in $\frak{g}=C_{m}$; and if $V$
is self-dual with a symmetric
invariant form, $\frak{r}$ is a maximal subalgebra in $\frak{g}=B_{m}$ or $%
D_{m}$. Denote by $o(V)$ the number of nilpotent orbits in $\frak{g}$, then
\begin{equation*}
o(V)\geq o_{n}=\left\{
\begin{array}{l}
\min (a_{n-1},b_{m})\text{, if }n\text{ is odd} \\
\min (a_{n-1},d_{m},c_{m})\text{, if }n\text{ is even}
\end{array}
\right. .
\end{equation*}

We want to check that for any irreducible $\frak{r}$ -module $V$ (except
those corresponding to the trivial embedding $\frak{r}$ $=\frak{g}$), the
number $o(\frak{r})$ of the nilpotent orbits in $\frak{r}$ is less than $%
o(V) $. In what follows, representations $\phi $ generating trivial
embeddings of $\frak{r}$ in $\frak{g}$ are excluded. For exceptional $\frak{r%
}$ $\ $of types $G_{2}$, $F_{4}$, $E_{6}$, $E_{7}$, $E_{8}$ the smallest
irreducible representation has dimension $n=7$, $26$, $27$, $56$, $248$
respectively. In all cases, the inequality $o(\frak{r})<o_{n}$ holds.

\bigskip If $\frak{r}$ is of type $A_{k}$, then either $V$ is not self-dual
and $n>k+1$ (in which case $o(\frak{r})=a_{k}<a_{n-1}=o(V)$) or $V$ is
self-dual and $n>2(k+1)$. Then \ $a_{k}<o_{n}$ $\leq o(V)$.

If $\frak{r}$ is of type $B_{k}$ ($k\geq 2$), then all irreducible $V$ are
self-dual and $n>2k+1$, hence $b_{k}<o_{n}$. If $\frak{r}$ is of type $C_{k}$
($k\geq 3$), then all irreducible $V$ are self-dual and $n>4k$, hence $%
c_{k}<o_{n}$. If $\frak{r}$ is of type $D_{k}$ ($k\geq 4$), then for any
irreducible $V$, $n>3k$ and $d_{k}<o_{n}$.

From this we conclude that Proposition \ref{prop.30} applies to any simple
non-regular $\frak{r}$ in $\frak{g}$, where $\frak{g}$ is a classical simple
Lie algebra.
\end{remark}

\begin{lemma}
\label{degen.1} Each orbit $O$ of a fixed point free degenerate $SL_2$%
-action on an affine algebraic variety $X$ is two-dimensional and closed,
and the isotropy group of any point $x \in X$ is either $\ensuremath{%
\mathbb{C}}^*$ or $\ensuremath{\mathbb{Z}}_2$-extension of $%
\ensuremath{\mathbb{C}}^*$ . 
\end{lemma}

\begin{proof} In the case of a fixed point free $SL_2$-action the isotropy group $I_x$ of a point
from a closed orbit is either finite or $\C^*$ or $\Z_2$-extension
of $\C^*$ by Lemma \ref{isosl2}. Because the action is also
degenerate $I_x$ cannot be finite and, and therefore the closed
orbit $SL_2/I_x$ is two-dimensional. By Proposition \ref{redu} (2)
the closure of $O$ must contain a closed orbit. Since $O$ itself
is at most two-dimensional it must coincide with this closed
orbit.

\end{proof}

Next we need two lemmas with the proof of the first one being straightforward.

\begin{lemma}
\label{dynkin.10} Let $G$ be a simple Lie group of dimension $N$ and rank $n$%
, $a$ be an element of $G$ and $C(a)$ be its centralizer. Suppose that $k$
is the dimension of $C(a)$. Then the dimension of the orbit $O$ of $a$ under
conjugations is $N-k$. In particular, when $a$ is a regular element (i.e. $%
\dim C(a)=n$) we have $\dim O=N-n$ coincides with the codimension of the
centralizer of $a$.
\end{lemma}

\begin{lemma}
\label{dynkin.20} Let $G$ be a simple Lie group of dimension $N$ and rank $n$%
, $R$ be its proper reductive subgroup of dimension $M$ and rank $m$, $%
\Gamma $ be a $SL_2$-subgroup of $G$ such that its natural action on $G/R$
is fixed point free degenerate. Suppose that $a$ is a semi-simple
non-identical element of $\Gamma$ and $k$ is the dimension of $C(a)$. Then $%
M\geq N-k-1$. Furthermore, if $a$ is regular $M=N-n+m-2$.
\end{lemma}

\begin{proof}
Since the $\Gamma$-action on $G/R$ is fixed point free and
degenerate 
the isotropy group of any element $gR \in G/R$ is either $\C^*$ or
a $\Z_2$ extension of $\C^*$ by Lemma \ref{degen.1}. Recall that
this isotropy group is $\Gamma  \cap gRg^{-1}$ and
therefore, $R$ contains a unique subgroup of form $g^{-1}L'g$
where $L'$ is a $\C^*$-subgroup of $\Gamma$. That is,
$L'=\gamma_0^{-1} L \gamma_0$  for some $\gamma_0 \in \Gamma$
where $L$ is the $\C^*$-subgroup generated by $a$. Furthermore,
this $\gamma_0$ is unique modulo a normalizer of $L$ in $\Gamma$
because otherwise $\Gamma^g$ contain another $\C^*$-subgroup of
$g^{-1}\Gamma g$ and, therefore, it would be at least two
dimensional. The two-dimensional variety $W_{a,g}=\{ (\gamma
g)^{-1} a (\gamma g) | \, \gamma \in \Gamma \}$ meets $R$ exactly
at two points $(\gamma_0 g)^{-1} a (\gamma_0 g)$ and $(\gamma_0
g)^{-1} a^{-1} (\gamma_0 g)$ (since the normalizer of $L$ has two
components). Varying $g$ we can suppose that $W_{a,g}$ contains  a
general point of the $G$-orbit $O_a$ of $a$ under conjugations.
Since it meets subvariety $R\cap O_a$ of $O_a$ at two points we
see that $\dim R\cap O_a =\dim O_a-2=N-k-2$ by Lemma
\ref{dynkin.10}. Thus, with $a$ running over $L$ we have $\dim R
\geq N-k-1$.

For the second statement note $b=g^{-1}ag\in R$ is a regular
element in $G$.
Hence the maximal torus in $G$ (and, therefore, in $R$) containing
$b$ is determined uniquely. Assume that two elements
$b_l=g_l^{-1}ag_l\in R, \, l=1,2$ are contained in the same
maximal torus $T'$ of $R$ and, therefore, the same maximal torus
$T$ of $G$. Then $g_2g_1^{-1}$ belongs to the normalizer of $T$,
i.e. $b_2$ is of form $w^{-1}b_1w$ where $w$ is an element of the
Weyl group of $T$. Thus $R\cap O_a$ meets each maximal torus $T'$
at a finite number of points. The space of maximal tori of $R$ is
naturally isomorphic to $R/T_{\rm norm}'$ where $T_{\rm norm}'$ is
the normalizer of $T'$. Hence $\dim R\cap O_a = \dim R - \dim
T_{\rm norm}'=M-m$. We showed already that the last dimension is
also $N-n-2$ which implies $M=N-n+m-2$.

\end{proof}

\begin{proposition}
\label{dynkin.30} Let the assumption of Lemma \ref{dynkin.20} hold. Then $a$
cannot be a regular element of $G$.
\end{proposition}

\begin{proof}
Assume the contrary, i.e. $a$ is regular. Let $\ggoth$ be the Lie
algebra of $G$, $\hgoth$ be its Cartan subalgebra, $\rgoth_+$
(resp. $\rgoth_-$) be the linear space generated by positive
(resp. negative) root spaces. Set $\sgoth =\rgoth_++\rgoth_-$ and
suppose that $\ggoth', \hgoth', \sgoth'$ be the similar objects
for $R$ with $\hgoth' \subset \hgoth$. Put
 $\rgoth_\pm'= \sgoth' \cap \rgoth_\pm$.

Each element of a root space $x'$ from $\sgoth'$ is of form
$x'=h_0+x_++x_-$ where $h_0\in \hgoth$ and $x_\pm \in \rgoth_\pm$.
Then there exists element $h' \in \hgoth'$ such that the Lie
bracket $[x' ,h']$ is a nonzero multiple of $x'$ which implies
that $h_0=0$, since $[h_0,h']=0$. Thus $\sgoth' \subset \sgoth$.
By assumption $\hgoth'$ is a linear subspace of $\hgoth$ of
codimension $n-m$. Hence Lemma \ref{dynkin.20} implies that
$\sgoth'$ is of codimension 2 in $\sgoth$. We have two
possibilities: (1) either, say, $\rgoth_+' =\rgoth_+$ and
$\rgoth_-'$ is of codimension 2 in $\rgoth_-$ or (2) $\rgoth_\pm'$
is of codimension 1 in $\rgoth_\pm$. In the first case each
element of a root space $x \in \rgoth_+$, being in an eigenspace
of $\hgoth' \subset \hgoth$, is also an element of a root space of
$\ggoth'$. However for each root the negative of it is also
contained in $\ggoth'$ which implies that $\rgoth_-' =\rgoth_-$. A
contradiction.

In case (2) consider the generators $x_1, \ldots , x_l$ (resp.
$y_1, \ldots , y_l$) of all root spaces in $\rgoth_+$ (resp.
$\rgoth_-$) such that $h_i=[x_i,y_i]$ is a nonzero element of
$\hgoth$. 
Their linear combination $\sum_{i=1}^lc_i^+x_i$ is contained in
$\rgoth_+'$ if and only if its coefficients satisfy a nontrivial
linear equation $\sum_{i=1}^ld_i^+c_i^+=0$. Similarly
$\sum_{i=1}^lc_i^-y_i$ is contained in $\rgoth_-'$ if and only if
its coefficients satisfy a nontrivial linear equation
$\sum_{i=1}^ld_i^-c_i^-=0$. Note that $d_i^+=0$ if and only if
$d_i^-=0$ since otherwise one can find a root of $\ggoth'$ whose
negative is not a root. Without loss of generality we suppose that
the simple roots are presented by $x_1, \ldots , x_{n}$, i.e.
$h_1, \ldots , h_{n}$ is a basis of $\hgoth$. Hence at least one
coefficient $d_i^+ \ne 0$ for $i \leq n$. Indeed, otherwise
$\rgoth_+'$ contains $x_1, \ldots , x_{n}$ which implies that $\rgoth_+'=
\rgoth_+$ contrary to our assumption.
Note that $[a, x_i]=2x_i$ for $i \leq n$ (\cite{Bo},
Chapter 8.11.4, Proposition 8). Furthermore, since any $x_j, \, j
\geq n+1$ is a Lie bracket of simple elements one can check via
the Jacobi identity that $[a,x_j]=s x_j$ where $s$ is an even
number greater than 2. If we assume that $d_j^+ \ne 0$ then a
linear combination $x_i+cx_j, \, c \ne 0$ is contained in
$\rgoth_+'$ for some $x_i, \, i \leq n$ . Taking Lie bracket with
$a$ we see that $2x_i+scx_j \in \rgoth_+'$. Hence $x_j \in
\rgoth_+'$, i.e. $d_j^+ = 0$ which is absurd. Thus $d_k^+ \ne 0$
only for $k \leq n$. We can suppose that $d_i^+ \ne 0$ for $i \leq
l_o \leq n$ and $d_j^+=0$ for any $j \geq l_o+1$. Note that $h_j
\in \hgoth'$. If $l_o \geq 3$ pick any three distinct numbers
$i,j$, and $k \leq l_o$. Then up to nonzero coefficients $x_i+x_j
\in \rgoth_+'$ and $y_i+y_k \in \rgoth_-'$. Hence
$h_i=[x_i+x_j,y_i+y_k] \in \hgoth'$, i.e. $\hgoth' =\hgoth$. In this case
we can find $h \in \hgoth'$ such that $[h,x_i]=s_ix_i$ and $[h,x_j]=s_jx_j$ with
$s_i \ne s_j$. As before this implies that $x_i \in \rgoth_+'$
which is a contradiction. Thus we can suppose that at most $d_1^+$
and $d_2^+$ are different from zero.

If $l_0\leq 2$ and $n\geq 3$
we can suppose that $[x_2,x_3]$ is a nonzero nilpotent element.
The direct computation shows that up to nonzero coefficients
 $[[x_2,x_3],[y_2,y_3]]$ coincides with $h_2-h_3$. Since $h_3 \in \hgoth'$, so is $h_2$.
The same argument works for $h_1$, i.e. $\hgoth' =\hgoth$ again which leads to a contradiction as before.
If $n= 2$ then the rank $m$ of $R$ is 1 (since we do not want $\hgoth'=\hgoth$), i.e. $R$ is either
$\C^*$ or $SL_2$. In both cases
$\dim R< \dim G -n+m-2$ contrary to
Lemma \ref{dynkin.20} which yields the desired conclusion.


\end{proof}

Combining this result with Definition \ref{principal} and Proposition \ref
{prin.10} we get the following.

\begin{corollary}
\label{dynkin.40} Let $G$ be a simple Lie group and $R$ be its reductive
non-principal subgroup. Then for the principal $SL_2$-subgroup $\Gamma < G$
we have $\Gamma^{g_0}$ is finite for some $g_0 \in G$ and $\Gamma^g$ is
different from $g^{-1}\Gamma g$ for any $g \in G$.
\end{corollary}

\begin{lemma}
\label{p1}Let $R$ be a principal subgroup of $G$. Then there exists an $%
SL_{2}$-subgroup $\Gamma <G$ such that $\Gamma ^{g_{0}}$ is finite for some $%
g_{0}\in G$ and $\Gamma ^{g}$ is different from $g^{-1}\Gamma g$ for any $%
g\in G$.\label{principal1}
\end{lemma}

\textbf{Proof. }Recall that the subregular nilpotent orbit is the
unique nilpotent orbit of codimension ${\rm rank}\, \frak{g}+2$ in
$\frak{g}$ \cite[Section 4.1]{CM}. It can be characterized as the
unique open orbit in
the boundary of the principal nilpotent orbit. The corresponding $\frak{sl}%
_{2}$-triple $(X,H,Y)$ in $\frak{g}$ is also called subregular.
The dimension of the centralizer of the semisimple subregular
element $H$ in this triple is ${\rm rank}\, \frak{g}+2$ \cite{CM}.
We denote the subregular $SL_{2}$ subgroup of $G$ by $\Gamma
_{sr}$.\\[2ex]

\begin{tabular}{|l|l|l|l|}
\hline $G$ & $R$ & $ {\rm rank} \, G+3$ & $\dim G-\dim R$ \\
\hline $B_{3}$
& $G_{2}$ & $6$ & $7$ \\ \hline $D_{4}$ & $G_{2}$ & $7$ & $14$ \\
\hline $A_{6}$ & $G_{2}$ & $9$ & $34$ \\ \hline $E_{6}$ & $F_{4}$
& $9$ & $26$ \\ \hline $A_{2l-1}$ & $C_{l}$ & $2l+2$ &
$l(2l-1)-1$ \\ \hline $A_{2l}$ & $B_{l}$ & $2l+3$ & $2l^{2}+3l$ \\
\hline $D_{l}$ & $B_{l-1}$ & $l+3$ & $2l-1$ \\ \hline
\end{tabular}\\[2ex]

We will demonstrate that in the cases listed in the table above, no
conjugates of $\Gamma _{sr}$ can belong to $R$. Then by Lemma \ref{dynkin.20}
the statement of the current Lemma follows whenever $\dim G-\dim R\frak{\,>}%
{\rm rank}\, G+3$. From the table above we see that it covers all
the principal embeddings from the proof of Proposition
\ref{prin.10}, with the exceptions of the inclusions $B_{3}\subset
D_{4}$ and $C_{2}\subset A_{3}$.

For $G=A_{r}$, the subregular $\frak{sl}_{2}$ corresponds to the partition $%
(r,1)$. If $r$ is odd, this partition is not symplectic (since in symplectic
partitions all odd entries occur with even multiplicity), and if $r$ is
even, this partition is not orthogonal (since in orthogonal partitions all
even entries occur with even multiplicity). In other words, the subregular $%
SL_{2}$-subgroup $\Gamma _{sr}$ in $A_{2l-1}$ (respectively, $A_{2l}$) does
not preserve any nondegenerate symplectic (resp., orthogonal) form on $%
\mathbb{C}^{2l}$ ($\mathbb{C}^{2l+1}$) and thus does not belong to $R=C_{l}$
(resp., $R=B_{l}$). The same is true for any conjugate of $\Gamma _{sr}$ in $%
G$.

If $G=D_{l}$, the embedding of $R=SO_{2l-1}$ in $SO_{2l}$ is defined by the
choice of the nonisotropic vector $v\in \mathbb{C}^{2l}$ which is fixed by $R
$. The subregular $\frak{sl}_{2}$ in $\frak{so}_{2l}$ corresponds to the
partition $(2l-3,3)$. Thus we see that $\Gamma _{sr}\subset SO_{2l}$ does
not fix any one-dimensional subspace in $\mathbb{C}^{2l}$ (its invariant
subspaces have dimensions $2l-3$ and $3$) and thus none of its conjugates
can belong to $R$. Moreover, we can choose $v$ such that $xv\neq v$ for $%
x\in \Gamma _{sr}$, $x\neq 1$. Thus $\Gamma _{sr}\cap $ $SO_{2l-1\text{ }%
}=\left\{ e\right\} $. This establishes the desired conclusion for
the embeddings $SO_{7}\subset SO_{8}$ (i.e. the case of $B_3
\subset D_4$) and $SO_{5}\subset SO_{6}$ (i.e. the case of $B_2
\simeq C_2 \subset A_3 \simeq D_3$), in which the dimension count
of Lemma \ref{dynkin.20} by itself is not sufficient.

The alignments of $\frak{sl}_{2}$-triples in exceptional cases were analyzed
in \cite{LMW}. In particular, it was observed there that any conjugacy class
of $\frak{sl}_{2}$-triples in $\frak{f}_{4}$ lifts uniquely to a conjugacy
class of $\frak{sl}_{2}$-triples in $\frak{e}_{6}$. \ Consulting the
explicit correspondence given in \cite[2.2]{LMW}, we observe that the
largest non-principal nilpotent orbit in $\frak{e}_{6}$ which has nonempty
intersection with $\frak{f}_{4}$ has codimension 10. This implies that no $%
\frak{sl}_{2}$-triple in $\frak{f}_{4}$ lifts to a subregular $\frak{sl}_{2}$
in $\frak{e}_{6}$. In other words, no conjugate of $\Gamma _{sr}\subset E_{6}
$ belongs to $F_{4}$.

When $R=G_{2}$, its embedding in $SO_{8}$ is defined by the triality
automorphism $\tau :SO_{8}\rightarrow SO_{8}$, with $R$ being a fixed point
group of this automorphism. Equivalently, $R=$ $SO_{7}\cap \tau \left(
SO_{7}\right) $. In particular, only those $\frak{sl}_{2}$-triples in $\frak{%
so}_{8}$ which are fixed under the triality automorphism belong to $\frak{g}%
_{2}$. Observe that the subregular $(5,3)$ $\frak{sl}_{2}$-triple is not
fixed by triality (cf. \cite[Remark 2.6]{LMW}). Similarly, the subregular$%
\frak{\ sl}_{2}$-triple in $\frak{so}_{7}$ corresponds to the partition $%
(5,1,1)$. Since the $\left( 5,1,1,1\right) $ $\frak{sl}_{2}$-triple in $%
\frak{so}_{8}$ is not invariant under triality, neither is subregular $\frak{%
sl}_{2}$ in $\frak{so}_{7}$. Thus no conjugates of $\Gamma _{sr}$ in $B_{3}$
or $D_{4}$ are fixed by $\tau $, and no conjugates of $\Gamma _{sr}$ belong
to $G_{2}$.

Finally, when $G=A_{6}$, the subregular triple in $A_{6}$ does not
belong to $B_{3}$ (see above), and thus none of its conjugates lie
in $G_{2}\subset B_{3}$. \hspace{2.5in} $\square $

\medskip

Now Theorem \ref{dynkin} follows immediately from  the combined statements
of Corollary \ref{dynkin.40} and Lemma \ref{p1}.

\begin{remark}\label{fabrizio.31} Note that we proved slightly more than required. Namely,
the $SL_2$-subgroup $\Gamma$ in Theorem \ref{dynkin} can be chosen either principal
or subregular.

\end{remark}

\providecommand{\bysame}{\leavevmode\hboxto3em{\hrulefill}\thinspace}


\begin{thebibliography}{KaMi}
\bibitem{A}  E.~Anders\'en, \emph{Volume-preserving automorphisms of $%
\ensuremath{\mathbb{C}}^n$}, Complex Variables Theory Appl. \textbf{14}
(1990), no. 1-4, 223--235.

\bibitem{AL}  E.~Anders\'en, L.~Lempert, \emph{On the group of holomorphic
automorphisms of $\ensuremath{\mathbb{C}}^n$}, Invent. Math. \textbf{110}
(1992), no. 2, 371--388.

\bibitem{BaCa1}  P.~Bala, R.~W.~Carter, \emph{Classes of unipotent elements
in simple algebraic groups. I.} Math. Proc. Cambridge Philos. Soc. 79
(1976), no. 1, 401--425.

\bibitem{BaCa2}  P.~Bala, R.~W.~Carter, \emph{Classes of unipotent elements
in simple algebraic groups. II.} Math. Proc. Cambridge Philos. Soc. 80
(1976), no. 1, 1--17.

\bibitem{Bo}  N.~Bourbaki, \emph{Elements of Mathematics. Lie Groups and Lie
Algebras. Chapters 7-9.} Springer, Berlin-Heidelberg-New York, 2005.

\bibitem{Ch}  C. ~Chevalley, \emph{Th\'eorie des groupes de Lie. Tome II.
Groupes alg\'ebriques}, (French) Actualit\'es Sci. Ind. no. 1152. Hermann $%
\& $ Cie., Paris, 1951. vii+189 pp.

\bibitem{CM}  D.\ H.\ Collingwood, W.\ M.\ McGovern, \emph{Nilpotent orbits
in semisimple Lie algebras}, Van Nostrand Reinhold Mathematics Series. Van
Nostrand Reinhold Co., New York, 1993.

\bibitem{D}  J.M. Drezet, \emph{Luna's slice theorem and applications}
Algebraic group actions and quotients, 39--89, Hindawi Publ. Corp., Cairo,
2004.

\bibitem{Dy}  E.B. Dynkin, \emph{Semisimple subalgebras of semisimple algebra%
}, Amer. Math. Soc. Transl. (2) \textbf{6} (1957), 111-244.

\bibitem{FKZ}  H.\ Flenner, S.\ Kaliman, M.\ Zaidenberg, \emph{Completions
of $\mathbb{C\sp *}$-surfaces}, Affine algebraic geometry, 149--201, Osaka
Univ. Press, Osaka, 2007.

\bibitem{FuHa}  W. ~Fulton, J. ~Harris, \emph{Representation Theory. A First
Course} Graduate Texts in Mathematics, 129, Readings in Mathematics,
Springer-Verlag, New York, (1991).


\bibitem{FR}  F. ~Forstneri\v c, J.-P.~Rosay, \emph{Approximation of
biholomorphic mappings by automorphisms of $\ensuremath{\mathbb{C}}^n$},
Invent. Math. \textbf{112} (1993), no. 2, 323--349.

\bibitem{G}  F.D. ~Grosshans, \emph{Algebraic Homogeneous spaces and
invariant theory}, Lecture Notes in Mathematics, Springer (1991).


\bibitem{HI}  A. ~Huckleberry, A. ~Isaev \emph{Infinite-dimensionality of
the autromorphism groups of homogeneous Stein manifolds}, preprint, 14 p.,
arXiv:0806.0693 (2008).

\bibitem{KK1}  S.~Kaliman, F.~Kutzschebauch, \emph{Density property for
hypersurfaces $uv=p ({\bar x})$}, Math. Zeit. \textbf{258} (2008), 115-131.

\bibitem{KK2}  S.~Kaliman, F.~Kutzschebauch, \emph{Criteria for the density
property of complex manifolds}, Invent. Math. \textbf{172} (2008), 71-87.

\bibitem{KaMi}  T.\ Kambayashi, M.\ Miyanishi, \emph{On flat
fibrations by the affine line}, Illinois J. Math. \textbf{22} (1978), no. 4,
662--671.

\bibitem{KoMa}  J. Koll\'ar, F. Mangolte, \emph{Cremona
transformations and homeomorphism of surfaces}, preprint, 2008, 30 p..



\bibitem{LMW}  J. M. Landsberg, L. Manivel, B. W. Westbury. \emph{%
Series of unipotent orbits}. Experimental Mathematics, 13(2004), 13-29

\bibitem{Mo}  G. ~D. ~Mostow, \emph{Fully reducible subgroups of algebraic
groups}, Amer. J. Math. \textbf{78} (1956), 200--221.

\bibitem{OV}  A. L. Onishchik, E. B. Vinberg, \emph{Lie Groups and Lie
Algebras III, }Springer-Verlag (1994)

\bibitem{Po}  V.\ L.\ Popov, \emph{On polynomial automorphisms of affine
spaces} (Russian) Izv. Ross. Akad. Nauk Ser. Mat. \textbf{65} (2001), no. 3,
153--174; translation in Izv. Math. \textbf{65} (2001), no. 3, 569--587.

\bibitem{PV}  V.L. Popov, E.B. Vinberg, \emph{Invariant theory}, Albebraic
Geometry IV, Springer-Verlag (1991).

\bibitem{Sch}  G.~Schwarz, \emph{The topology of algebraic quotients}, in
\emph{Topological Methods in Algebraic Transformation Groups}, Birkh?user,
Boston-Basel-Berlin, 1989, pp. 135-151.

\bibitem{TV1}  A. ~Toth, D. ~Varolin, \emph{Holomorphic diffeomorphisms of
complex semisimple Lie groups}, Invent. Math. \textbf{139} (2000), no. 2,
351--369.

\bibitem{TV2}  A. ~Toth, D. ~Varolin, \emph{Holomorphic diffeomorphisms of
semisimple homogenous spaces}, Compos. Math. \textbf{142} (2006), no. 5,
1308--1326.

\bibitem{V1}  D. ~Varolin, \emph{The density property for complex manifolds
and geometric structures}, J. Geom. Anal. \textbf{11} (2001), no. 1,
135--160.

\bibitem{V2}  D. ~Varolin, \emph{The density property for complex manifolds
and geometric structures. II}, Internat. J. Math. \textbf{11}
(2000), no. 6, 837--847.

 \bibitem{W} J. ~Winkelmann, {\em Invariant rings and quasiaffine quotients},
 Math. Z. \textbf {244} (2003), no. 1, 163--174.
\end{thebibliography}
\end{document}